\documentclass[reqno,10pt]{amsart}
\usepackage{amsmath}
\usepackage{amssymb}
\usepackage[left=1in,right=1in]{geometry}

\newtheorem{theorem}{Theorem}
\newtheorem{lemma}{Lemma}

\newtheorem{proposition}{Proposition}
\newtheorem{definition}{Definition}
\newtheorem{remark}{Remark}

\newtheorem{corollary}{Corollary}

\usepackage{amsmath}
\usepackage{amssymb}
\usepackage{amsfonts}
\usepackage{enumerate}
\usepackage{comment}
\usepackage{braket}

\usepackage[T1]{fontenc}
\usepackage[mathscr]{eucal}

\usepackage{acronym}
\usepackage{latexsym}
\usepackage{paralist}
\usepackage{wasysym}
\usepackage{xspace}

\usepackage[dvipsnames,svgnames]{xcolor}
\colorlet{MyBlue}{DodgerBlue!60!Black}
\colorlet{MyGreen}{DarkGreen!85!Black}

\usepackage[font=small,labelfont=bf]{caption}
\usepackage{subfigure}
\usepackage{tikz}
\usetikzlibrary{calc,patterns}

\usepackage{hyperref}
\hypersetup{
	colorlinks=true,
	linktocpage=true,
	pdfstartview=FitH,
	breaklinks=true,
	pdfpagemode=UseNone,
	pageanchor=true,
	pdfpagemode=UseOutlines,
	plainpages=false,
	bookmarksnumbered,
	bookmarksopen=false,
	bookmarksopenlevel=1,
	hypertexnames=true,
	pdfhighlight=/O,
	urlcolor=MyBlue!60!black,linkcolor=MyBlue!70!black,citecolor=DarkGreen!70!black, 
	pdftitle={},
	pdfauthor={},
	pdfsubject={},
	pdfkeywords={},
	pdfcreator={pdfLaTeX},
	pdfproducer={LaTeX with hyperref}
}

\numberwithin{equation}{section}  
\usepackage[sort&compress,capitalize,nameinlink]{cleveref}
\crefname{app}{Appendix}{Appendices}

\crefrangeformat{equation}{\upshape(#3#1#4)\textendash(#5#2#6)}

\usepackage[textwidth=30mm]{todonotes}

\usepackage{soul}
\setstcolor{red}
\sethlcolor{SkyBlue}

\newcommand{\x}{\mathbf{x}}

\newcommand{\whp}{\textbf{whp}}
\newcommand{\ind}{\mathbf{1}}

\newcommand{\cA}{\ensuremath{\mathcal A}} 
\newcommand{\cB}{\ensuremath{\mathcal B}} 
 
\newcommand{\cD}{\ensuremath{\mathcal D}} 
 
\newcommand{\cF}{\ensuremath{\mathcal F}}

\newcommand{\cK}{\ensuremath{\mathcal K}} 
\newcommand{\cL}{\ensuremath{\mathcal L}}

\newcommand{\E}{\ensuremath{\mathbb{E}}}

\newcommand{\N}{\ensuremath{\mathbb{N}}}

\newcommand{\R}{\ensuremath{\mathbb{R}}}
\renewcommand{\P}{\ensuremath{\mathbb{P}}}

\def\({\left(}
\def\){\right)}
\def\[{\left[}
\def\]{\right]}
\newacro{NE}{Nash equilibrium}
\newacroplural{NE}[NE]{Nash equilibria}
\newacro{PNE}{pure Nash equilibrium}
\newacroplural{PNE}[PNE]{Nash equilibria}
\newacro{PFNE}{prior-free Nash equilibrium}
\newacroplural{PFNE}[PFNE]{prior-free Nash equilibria}
\newacro{WE}{Wardrop equilibrium}
\newacroplural{WE}[WE]{Wardrop equilibria}
\newacro{SO}{socially optimum}

\newacro{KKT}{Karush\textendash Kuhn\textendash Tucker}
\newacro{OD}[O/D]{origin-destination}
\newacro{PoA}{price of anarchy}
\newacro{PoS}{price of stability}
\newacro{PoCS}{price of correlated stability}
\newacro{BPR}{bureau of public roads}
\newacro{FIP}{finite improvement property}

\newacro{BPG}{buck-passing game}
\newacro{SBPG}{stochastic buck-passing game}
\newacro{MBPG}{mixed extension of the buck-passing game}

\begin{document}

\title[Loop-erased partitioning ]{Loop-erased partitioning of a graph:\\ mean-field analysis}

\author[L.~Avena]{Luca Avena$^\ddag$} 
\address{{$^\ddag$ Leiden University, Mathematical Institute, Niels Bohrweg 1
		2333 CA, Leiden. The Netherlands.}}
\email{l.avena@math.leidenuniv.nl}

\author[A.~Gaudilli\`ere]{Alexandre Gaudilli\`ere$^\star$} 
\address{
	$\star$ Aix-Marseille Universit\'e, CNRS, Centrale Marseille. I2M UMR CNRS 7373. 39, rue Joliot Curie. 13 453 Marseille Cedex
13. France.}
\email{alexandre.gaudilliere@math.cnrs.fr}

\author[P.~Milanesi]{Paolo Milanesi$^\S$} 
\address{
	$^\S$ Aix-Marseille Universit\'e, CNRS, Centrale Marseille. I2M UMR CNRS 7373. 39, rue Joliot Curie. 13 453 Marseille Cedex
	13. France.}
\email{paolo.milanesi@univ-amu.fr}

\author
[M.~Quattropani]
{Matteo Quattropani$^*$}
\address{$^*$ Dipartimento di Matematica e Fisica, Universit\`a di Roma Tre, Largo S. Leonardo Murialdo 1, 00146 Roma, Italy.}
\email{matteo.quattropani@uniroma3.it}


\begin{abstract} We consider a random partition of the vertex set of an arbitrary graph that can be sampled using loop-erased random walks stopped at a random independent exponential time of parameter $q>0$, that we see as a tuning parameter.The related random blocks tend to cluster nodes visited by the random walk on time scale $1/q$.
We explore the emerging macroscopic structure by analyzing 2-point correlations. 
To this aim, it is defined an interaction potential between pair of vertices, as the probability that they do not belong to the same block of the random partition.
This interaction potential can be seen as an affinity measure for ``densely connected nodes'' and capture well-separated regions in network models presenting non-homogeneous landscapes. 
In this spirit, we compute this potential and its scaling limits on a complete graph and on a non-homogeneous weighted version with community structures. For the latter geometry we show a phase-transition for ``community detectability'' as a function of the tuning parameter and the edge weights.
\end{abstract}

\maketitle
\section{\large{Intro: Loop-erasure and random partitioning}}\label{intro}
Consider an arbitrary simple undirected weighted connected graph $G=(V, E, w)$ on $N=|V|$ vertices where $E=\{e=(x,y): x,y \in V \}$ stands for the edge set and $w: E \rightarrow [0,\infty)$ is a given edge-weight function.
We call the Random Walk (RW) associated to $G$ the continuous-time Markov chain $X=(X_t)_{t\ge 0}$ with state space $V$ and \emph{the discrete Laplacian} as infinitesimal generator, i.e., the $N\times N$ matrix:
\begin{equation}\label{Laplacian} \mathcal{L}=  \mathcal{A}- \mathcal{D}, \end{equation} 
where for any $x,y\in [N]:=\{1,2,\dots ,N\}$,  $\mathcal A(x,y)=w(x,y)\mathbf{1}_{\{ x\neq y\}}$ is the \emph{weighted adjacency matrix} and  $\mathcal D(x,y)=\mathbf{1}_{\{ x=y\}}\sum_{z\in[N]\setminus\{x\}} w(x,z) $ is the \emph{diagonal matrix} guarantying that the entries of each row in $\mathcal L$ sum up to $0$.

The goal of this paper is to explore the following probability measure on the set of partitions $\mathcal P(V)$ of the vertex set $V$.
\begin{definition}[{\bf Loop-erased partitioning}]\label{LEP}
	Given $G=(V, E, w)$, fix a positive parameter $q>0$. We call \emph{loop-erased} a partition of $V$ in $m\leq N$ blocks sampled according to the following probability measure:
	\begin{equation}\label{LEPmeas}
	\mu_q(\Pi_m)=  \frac{q^m \times \sum_{F: \Pi(F)=\Pi_m } w(F)}{Z(q)}, \quad\quad \Pi_m\in \mathcal P(V), \end{equation}
	where the sum is over spanning rooted forests $F$'s of $G$, 
	$\Pi(F)$ stands for the partition of $V$ induced by a forest $F$, $w(F):=\prod_{e\in F} w(e)$ for the forest weight, and 
	$Z(q)$ is a normalizing constant.
	We denote by $\Pi_q$ a random variable in $\mathcal P(V)$ with law $\mu_q$. 
\end{definition}
In the above definition a spanning rooted forest of a graph is a collection of rooted trees spanning its vertex set. Denoting by $\mathcal F$ the set of spanning rooted forests of $G$, we note that---due to the matrix tree theorem---the normalizing constant in~\cref{LEPmeas} can be expressed as the characteristic polynomial of the matrix $\mathcal L$ evaluated at $q$, i.e. 
$$Z(q):=\sum_{F\in \mathcal F }q^{|F|} w(F)=\det[qI- \mathcal L],$$
where $|F|$ denotes the number of trees in $F\in \mathcal F$. Furthermore, the number of blocks in $\Pi_q$, denoted by $|\Pi_q|$, is distributed as the sum of $N$ independent Bernoulli random variables with success probabilities $\frac{q}{q+\lambda_i}$, for $i\leq N $, with $\lambda_i$'s being the eigenvalues of $- \cL$.
We refer the reader to~\cite[Prop. 2.1]{AG} for a proof of these statements.

\subsection{Tuning parameter and underlying geometry.}
The first factor $q^m$ in~\cref{LEPmeas} favors partitions having many small blocks as $q$ growths, while 
as $q$ vanishes, the measure degenerates into a one-block partition. 
The second combinatorial factor takes into account the underlying geometry and for example in the unweighted case (i.e. constant edge--weights $w\equiv1$ ) counts how many rooted forests are compatible with a given partition. In the simple setup of an unweighted complete graph on $N$ vertices , the measure in~\cref{LEP} reduces to  
\begin{equation}\label{completeLEP}
\mu_q(\Pi_m)=  \frac{q^m \times \prod_{i=1}^m n_i^{n_i-1}}{q(q+N)^{N-1}}, \end{equation}
for a partition $\Pi_m=\{B_1,\ldots, B_m\}\in \mathcal P(V)$ constituted of $m$ blocks with sizes $|B_i|=:n_i$, $i\leq m$ such that $\sum_{i\leq m} n_i=N$. 
In particular, we see in this setup that this second factor favors partitions with a few ``fat'' blocks. 
Notice that~\cref{completeLEP} holds true because, by Cayley's formula,  $n_i^{n_i-2}$ unrooted trees can cover block $B_i$, and since we are dealing with rooted trees, an extra volume factor $n_i$ for the possible roots is needed.
In general, the competition between these two factors depends on the delicate interplay among the tuning parameter $q$, the underlying geometry and the weight function $w$.

%

\subsection{Sampling algorithm and Loop-Erased RW (LERW)}\label{Wilson}
An attractive feature of this measure is that there exists a simple 
exact sampling algorithm. Originally due to Wilson~\cite{W96} and based on the associated LERW killed at random times. 
The LERW with killing is the process obtained by running the RW $X$, erasing cycles as soon as they appear, and stopping the evolving self-avoiding trajectory at an independent random time $\tau_q$ with law an exponential of parameter $q$. 
  
The algorithm can be described as follows:
\begin{enumerate}
	\item\label{1} pick \emph{any arbitrary} vertex in $V$ and run a LERW up to an independent time $\tau_q\overset{d}{\sim}\exp(q).$ Call $\gamma_1$ the obtained self-avoiding trajectory. 
	\item\label{2} pick \emph{any arbitrary} vertex in $V$ that does not belong to $\gamma_1$. 
	Run a LERW until $\min\{\tau_q, \tau_{\gamma_1}\}$, $\tau_{\gamma_1}$ being the first time the RW hits a vertex in $\gamma_1$. Call $\gamma_2$ the union of $\gamma_1$ and the new self-avoiding trajectory obtained in this step. 
	\item Iterate step (\ref{2}) with $\gamma_{\ell+1}$ in place of $\gamma_{\ell}$ until exhaustion of the vertex set $V$. 
\end{enumerate}  

In step (\ref{2}) we note that if the killing occurs before $\tau_{\gamma_1}$, then $\gamma_2$ is a rooted forest in $G$, else $\gamma_2$ is a rooted tree.
 
When the above algorithm stops, it produces a \emph{spanning rooted forest} $F\in \mathcal F$, where the roots are the points where the involved LERWs were killed along the algorithm steps.
The resulting forest $F$ on $G$ induces the partition $\Pi(F)$ of the vertex set $V$, where each block is identified by vertices belonging to the same tree. 
It can be shown that the probability to obtain a given rooted spanning forest $F$ is proportional to $q$ to the power of the number of trees, times the forest weight $w(F)$. It then follows that the induced partition is distributed as $\Pi_q$ in~\cref{LEP}. We refer the reader to~\cite{AG} for the proof of the latter and for more detailed aspects of this algorithm, including dynamical variants. In the sequel we will denote by $\P$ a probability measure on an abstract probability space sufficiently rich for the randomness required by this algorithm.  

\subsection{Partition detecting ``metastable landscapes''.}  
The Wilson's sampling algorithm described above shows that the resulting partition has the tendency to cluster 
in the same block (tree) points that can be visited by the RW with high probability on time scale $\tau_q$.
In this sense the loop-erased partitioning has the tendency to capture \emph{metastable-like regions} (blocks), namely, regions of points from which it is difficult for the RW to escape on time scale $1/q$.
This makes the probability $\mu_q$ an interesting measure for randomized clustering procedures, see in this direction~\cite{ACGM1} and~\cite[Sec. 5]{ACGM2}.
Yet, a-priori it is not clear how strong and stable is this feature of capturing ``metastable landscapes'', since it heavily depends on the underlying geometry (weighted adjacency matrix) and the choice of the killing parameter $q$.
The scope of this paper is to start making precise this heuristic by analyzing 2-points correlations associated to $\mu_q$ on the simplest \emph{dense} informative geometries.

\subsection{Two-point correlations}
For a pair of distinct vertices $x,y\in V$, consider the event that these vertices belong to different blocks in $\Pi_q$. That is, the event $$\{B_q(x)\neq B_q(y) \}:=\{x \text{ and } y \text{ are in different blocks of } \Pi_q\},$$ where $B_q(z)$ stands for the block in $\Pi_q$ containing $z\in V$. 
The probability of this event induces a 2-point correlation function which turns out to be analyzable by means of LERW explorations, and it encodes relevant information on how the resulting partition looks like on the underlying graph as a function of the parameters. 
Here is the formal definition together with an operative characterization.
\begin{definition}[{\bf Pairwise LEP-interaction potential}]\label{FIP}
	For given $q>0$ and $G$, and any pair $x,y\in V$, we call \emph{pairwise LEP-interaction potential} the following probability: 
	\begin{align}\notag
	U_q(x,y):=&\P
	(B_q(x)\neq B_q(y))\\ &=\sum_{\gamma}\P^{LE
		_q}_x(\Gamma=\gamma)\P_y(\tau_\gamma>\tau_q)\label{LEdec}
	\end{align}
	where $\P_x^{LE_q}$ and $\P_x$ stand for the laws of the LERW killed at rate $q$ and of the RW, respectively, starting from $x\in V$, and 
	the above sum runs over all possible self-avoiding paths $\gamma$'s starting at $x$. 
\end{definition}

The representation in~\cref{LEdec} is a consequence of Wilson's sampling procedure described in~\cref{Wilson} and it holds true since, remarkably, in steps (\ref{1}) and (\ref{2}) of the algorithm the starting points can be chosen arbitrarily. 

Furthermore, we notice that, as for any generic random partition of $V$, such an interaction potential defines a distance on the vertex set. 
This specific metric $U_q(x,y)$ can be interpreted as an affinity measure capturing how densely connected vertices $x$ and $y$ are in the graph $G$. Thus providing a further motivation to analyze it.

Still, the observable captured by $U_q(x,y)$ is not the only one inducing a natural notion of 2-point correlations associated to $\Pi_q$. For example, if we express the LEP-potential in Definition~\ref{FIP} as an expectation, i.e.
$U_q(x,y)= \E\left[\mathbf{1}_{\{B_q(x)\neq B_q(y)\}}\right]$, we may think of normalizing it with the masses of the related blocks and obtain another natural 2-point correlation function. This is captured in the following definition.

\begin{definition}[{\bf Pairwise RW-interaction potential}]\label{RWIP}
	For given $q>0$ and $G$, and any pair $x,y\in V$, we call \emph{pairwise RW-interaction potential} the following correlation function: 
	\begin{align}\notag	
	\overline{U}_q(x,y):=\E\left[\frac{\mathbf{1}_{\{B_q(x)\neq B_q(y)\}}}{\mu(B_q(x)) \mu(B_q(y))}\right],
	\end{align}
	where $\mu(\cdot)$ is the uniform measure on $V$.

\end{definition}

As we will see, the functional $\overline{U}_q$ is actually much simpler to analyze but it captures less insightful information on the underlying graph structure. Further, unlike $U_q$, this is not a probability, it is neither a metric, and it does not allow to derive a description of the macroscopic structure of $\Pi_q$. In a sense, the latter is not surprising, in fact (see Lemma~\ref{RWexpress}) this alternative correlation function can be expressed in terms of the sole RW Green's kernel without need to introduce the LEP $\Pi_q$. Note in particular that the uniform measure $\mu$ in Definition~\ref{RWIP}	 corresponds to the invariant measure of the RW $X$.


\subsection{Related literature}
Several properties of the forest measure associated to the loop-erased partitioning have been derived in the recent~\cite{AG,AG1}. Based on these results, in~\cite[Prop. 6]{ACGM2} and~\cite[Sect. 5.2]{ACGM3}, the authors proposed an approach making use of the loop-erased partitioning and so-called intertwining dualities to describe the evolution of \emph{local equilibria} of a finite state space Markov chain. 


As mentioned before, this sampling method based on LERW is originally due to Wilson~\cite{W96} and shows that the measure considered herein is intimately related to the well-known \emph{Uniform Spanning Tree} (UST) measure. Actually the measure on spanning rooted forests mentioned in~\cref{Wilson} can be seen as a generalized version of the UST measure which is recovered by taking $q\downarrow 0$ when $w\equiv1$.
Therefore the results presented in this manuscript are along the line of the 
flourishing literature on statistical properties of the UST and LERW, see e.g.~\cite{A91,BK05,BP93,G80,K07,LS19,LSW04,P91,Pitman02,S00,S09,S19}.

A detailed exact and asymptotic analysis of observables related to Wilson's algorithm on a complete graph have been pursued in~\cite{P02}. The derivation of our results is in this spirit, although we deal with the additional randomness given by the presence of the killing parameter, which in turns makes the combinatorics more involved. 

We further mention that in dense geometries, the UST has been studied under the perspective of the continuous random tree topology on the complete graph~\cite{A91} and with respect to local weak convergence still on the complete graph~\cite{G80} and more recently on growing expanders admitting a limiting graphon\cite{HNT18}. These other interesting lines of investigation could also be naturally considered for the forest measure in~\cref{Wilson} but we will not pursue these approaches in this work. 

\subsection{Paper overview}
Our main theorems are presented in~\cref{results} and identify the LEP-potential in~\cref{FIP} and its asymptotics on a complete graph,~\cref{proporso}, and on a non-homogeneous complete graph with two communities,~\cref{2par2comssintpot,phasetrans}. Some consequences on the macroscopic emergent partition $\Pi_q$ on these mean-field models are derived in~\cref{macro}. The last result in~\cref{Rwdetection} concerns the asymptotics detectability related to the other 2-point correlation function in~\cref{RWIP}.
The concluding~\cref{proofcomplete,proof2com} are devoted to the proofs for the complete graph and the community model, respectively. 

\subsection{Basic standard notation}
In what follows we will use the following standard asymptotic notation.
For given positive sequences $f(N)$ and $g(N)$, we write:  
\begin{itemize}
	\item $f(N)=o(g(N))$ if $\lim_{N\to\infty}\frac{f(N)}{g(N)}=0$.
	\item $f(N)=O(g(N))$ if $\lim\sup_{N\to\infty}\frac{f(N)}{g(N)}<\infty$.
	\item $f(N)=\omega(g(N))$ if $\lim_{N\to\infty}\frac{f(N)}{g(N)}=\infty$.
	\item $f(N)=\Omega(g(N))$ if $\lim\inf_{N\to\infty}\frac{f(N)}{g(N)}>0$.
	\item $f(N)=\Theta(g(N))$ if $0< \lim\inf_{N\to\infty}\frac{f(N)}{g(N)}\le \lim\sup_{N\to\infty}\frac{f(N)}{g(N)}<\infty$.
	\item $f(N)\sim g(N)$ if $\lim_{N\to\infty}\frac{f(N)}{g(N)}=1$.
\end{itemize}
For $k\leq n\in \N$ we will denote by $(n)_{k}:=n(n-1)(n-2)\cdots(n-k)$ the descendent factorial.
Furthermore, we denote by $I$ the identity matrix, $\mathbf{1}$ and $\mathbf{1}'$, respectively, for the row and column vectors of all $1$'s, where the dimensions will be clear from the context. We will write $A^{Tr}$ for the \emph{transpose} of a matrix $A$.

\section{\large{Results: correlations and emerging partition on mean-field models}}\label{results}

Our first result characterizes the LEP-potential in absence of geometry for finite $N$, and shows that this probability is asymptotically non-degenerate at scale $\sqrt{N}$:  
\begin{theorem}\label{proporso}{\bf (Mean-field LEP-potential and limiting law)}
	Fix $q>0$ and let $\mathcal{K}_N$ be a complete graph on $N\geq 1$ vertices with constant edge weight $w>0$. Then, for all $x\neq y \in [N]$, 
	\begin{equation}\label{orsoformula}
	U^{(N)}_q(x,y)=U^{(N)}_q=\sum_{h=1}^{N-1}\frac{q}{q+Nw}\left(\frac{Nw}{q+Nw}\right)^{h-1}\prod_{k=2}^{h}\left(1-\frac{k}{N}\right),
	\end{equation}
	Furthermore, if $q=z\cdot w \sqrt{N}$, for fixed $z,w>0$, then 
	\begin{equation}\label{orsolimite}U_{q}:=\lim_{N\to\infty}U^{(N)}_{q}=\sqrt{2\pi}ze^{\frac{z^2}{2}}\P(Z>z),\end{equation} with $Z$ being a standard Gaussian random variable. 
\end{theorem}
Notice that the critical scale $\sqrt{N}$ is the typical length of a LERW path with no killing and---as can be derived by the results in~\cite{P02}---is the typical length of the first branch of the Wilson's algorithm on the complete graph, when $q=O(\sqrt{N})$. 

Our second result is the analogous of~\cref{orsoformula} when still every vertex is accessible from any other, but the edge weights are non-homogeneous and give rise to a community structure. In this sense we will informally refer to this network as of a \emph{mean-field-community} model. 
Formally, for given positive reals $w_1$ and $w_2$, we denote by $\mathcal{K}_{2N}(w_1,w_2)$ the graph $G$ with $V=[2N]$, and $w(e)=w_1$ if $e=(x,y)$ is such that either $x,y\in[N]$ or 
$x,y\in[2N]\setminus[N]$, and $w(e)=w_2$ otherwise. Thus, the weight $w_1$ measures the pairwise connection intensity within the same community, while $w_2$ between pairs of nodes belonging to different communities.
Given the symmetry of the model, we will use the notation $U^{(N)}_{q}(out)$ to refer to the potential $U^{(N)}_{q}(x,y)$, for $x$ and $y$ in different communities. Conversely, we set $U^{(N)}_{q}(in)$ for the potential associated to two nodes belonging to the same community.

\begin{theorem}\label{2par2comssintpot}{\bf (LEP-potential for mean-field-community model) }
	Fix $q, w_1, w_2>0$ and consider a two-community-graph $\mathcal{K}_{2N}(w_1,w_2)$. Let $T_q\geq1$ be a geometric random variable with success parameter 
	$$\xi_{q,N}:=\frac{q}{q+N(w_1+w_2)}$$
	and 
	let $\tilde X=\(\tilde X_n\)_{n\in\N_0}$ be a discrete-time Markov chain with state space $\{\underline{1},\underline{2}\}$ and transition matrix
	$$\tilde P=\left(\begin{matrix}p&1-p\\1-p&p \end{matrix}\right),\quad p=\frac{w_1}{w_1+w_2}.$$ Denote by 
	$\ell(n)=\sum_{m<n}\ind_{\left\{\tilde X_m=\underline{1} \right\}}$ the corresponding local time in state $\underline{1}$ up to time $n$ and by $\tilde \P_{\underline{1}}$ the corresponding path measure starting from $\underline{1}$.  
	
	For $x \in[N]$, set  $\star= in$ if $y\in[N]$, and $\star=out$ if  $y\in[2N]\setminus[N]$,then 
	
	\begin{equation}
	\begin{aligned}\label{g}
	U^{(N)}_{q}(x,y)=	U^{(N)}_{q}(\star)
	:=
	\sum_{n\geq 1} \P(T_q=n)\sum_{k= 1}^{n}  \tilde P_{\underline{1}}(\ell(n)=k) N^{-n+1}\hat{f}(n,k)\theta(n,k)P^{\dagger}_{\star}(n,k)
	\end{aligned}
	\end{equation}
	where
	\begin{equation}\label{fandtheta}
	\hat{f}(n,k)= (N-2)_{k-1}(N-1)_{n-k},\quad\quad
	\theta(n,k)=
	\frac{\left(q-\lambda_1(n,k)\right)\left(q-\lambda_2(n,k)\right)}{q(q+2Nw_2)}\end{equation} 
	
	with, for $i=1,2$,
	
	\begin{equation}
	\lambda_{i}(n,k)=-\frac{1}{2}\left[w_1n+w_2N+(-1)^{i}\sqrt{w_1^2(2k-n)^2+4\left(N-k\right)\left(N-k\right)w_2^2}\right],
	\end{equation}
	
	and 
	
	\begin{equation}\label{Pmorte}
	P^{\dagger}_{\star}(n,k)=\frac{ q( q+k_\star(w_1-w_2)+w_2N)}
	{[q+k w_1][q+(n-k)w_1]+Nw_2(2q+nw_1)+w_2^2[Nn-k(n-k)]} \times \eta_\star\end{equation}
	
	with 
	
	\begin{equation}
	k_\star:=\begin{cases}
	k, & \text{ if } \star = out, \\
	n- k, & \text{ if } \star = in,
	\end{cases}
	\quad
	\quad\quad\quad
	\eta_\star=\begin{cases}
	(N-1)(N-n+k-1), & \text{ if } \star = out, \\
	N(N-k-1), &  \text{ if } \star = in.
	\end{cases}
	\end{equation}
\end{theorem}

The above theorem is saying that the pairwise LEP-potential can be seen as the double-expectation of the function $g_{\star}(n,k)=N^{-n+1}
\left(\hat{f}\theta P^{\dagger}_{\star}\right)(n,k)$ in~\cref{g} with respect to the geometric time $T_q$ and to the local time of the coarse-grained RW $\tilde X$.
As can be seen in the proof, the analysis of this model can be in fact reduced to the study of such a coarse-grained RW jumping between the two ``lumped communities'' up to the independent random time $T_q$. The function $g_{\star}$ is the crucial combinatorial term encoding in the different parameter regimes the most likely trajectories for such a stopped two-state macroscopic walk $\tilde X$.

\begin{remark}{\bf (Extensions to many communities of arbitrary sizes and weigths) }
	The formula in~\cref{g} can be derived also for the general model with arbitrary number of communities of variable compatible sizes and arbitrary weights within and among communities. The corresponding statement and proof are more involved but they follow exactly the same scheme of this equal-size-two-community case captured in the above theorem. We refer the reader interested in such an extension to~\cite{Q16}.
\end{remark}

The next theorem gives the limit of the LEP-potential computed in~\cref{2par2comssintpot}, 
the resulting scenario is summarized in the phase-diagram in~\cref{fig:phdiag}. 

\begin{figure}[h]
	\includegraphics[width=14cm]{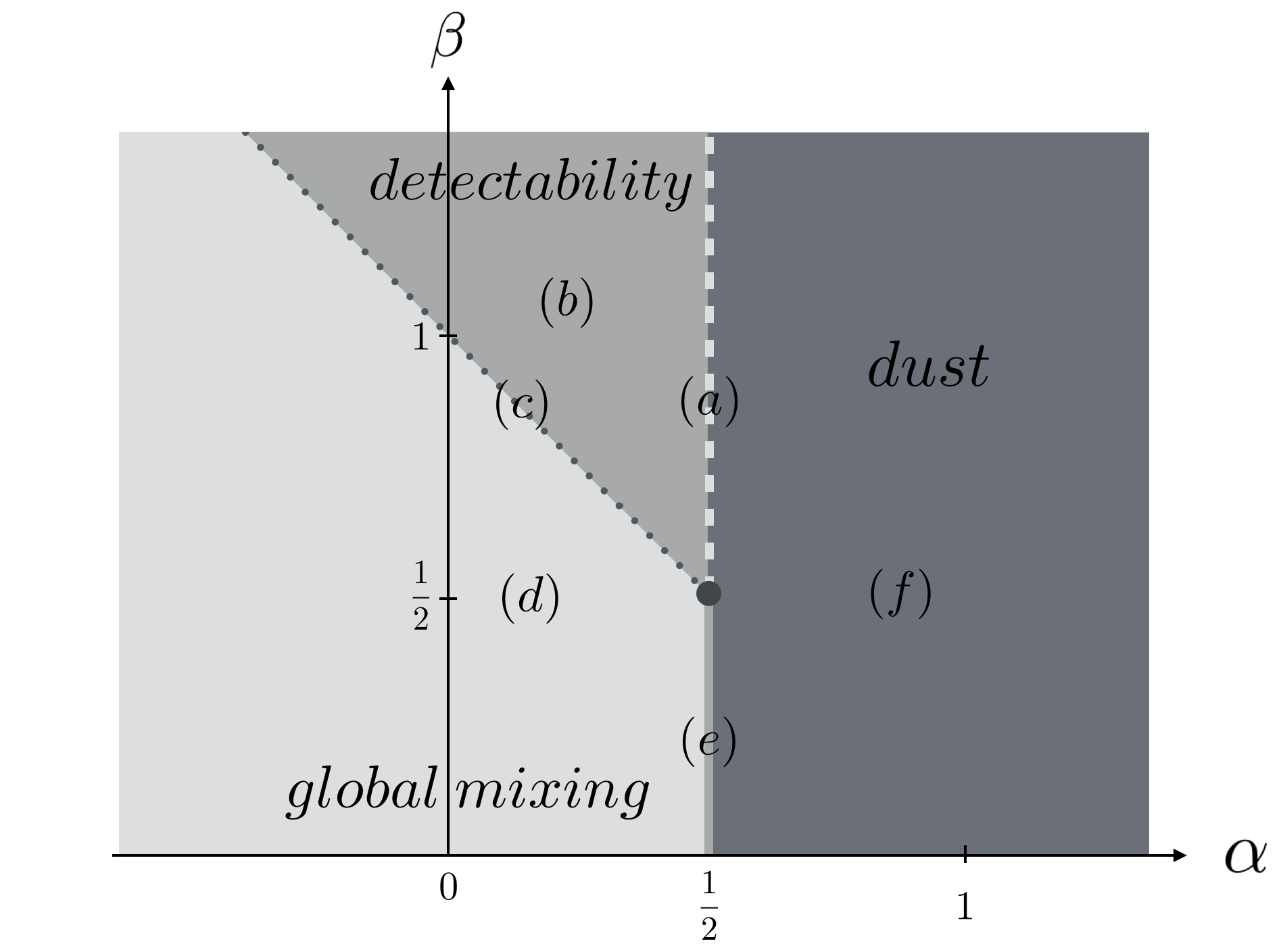}
	\caption{ $\alpha$--$\beta$ axis, $\alpha$ controls the killing rate ($q=N^\alpha$) and $\beta$ the weight between communities ($w_2=N^{-\beta}$).  	
	The above diagram describes at glance the limiting behavior of the LEP-potential as captured in~\cref{phasetrans}. 
The \emph{detectability} region (b) corresponds to the regimes where the difference of the \emph{in}- and \emph{out}-potential is maximal. In this case, indeed, the RW does not manage to exit its starting community within time scale $1/q$ and hence it is confined with high probability to ``its local universe''. In the \emph{dust} region (f) both \emph{in}- and \emph{out}-potential degenerates to 1, it is in fact a regime where the killing rate is sufficiently large (recall from~\cref{orsolimite} that $\sqrt{N}$ is the critical scale for the complete graph) to produce ``dust'' as emerging partition.
Finally, the \emph{global mixing} region (d) is the other degenerate regime where the RW ``mixes globally'' in the sense that it changes community many times within time scale $1/q$, hence loosing memory of its starting community. The separating lines (c)--(a)--(e) correspond to the delicate critical phases where the competition of the above behaviors occurs. This will become transparent in the proof in~\cref{detect} where such boundaries will deserve a more detailed asymptotic analysis.}
	\label{fig:phdiag}
\end{figure}

\begin{theorem}\label{phasetrans}{\bf (Detectability and phase diagram for two communities) }
	Under the assumptions of~\cref{2par2comssintpot}, set $w_1=1$ , $w_2=N^{-\beta}$ and $q=N^\alpha$ for some $\alpha\in\R,\: \beta\in\R^+$. 
	Then:
	\begin{itemize}
		\item[\bf{(a)}] if $1-\beta<\alpha=\frac{1}{2}$, $\lim_{N\to\infty}U^{(N)}_q(out)=1$ and $\lim_{N\to\infty}U^{(N)}_q(in)= \varepsilon_0(\beta)\in(0,1)$.
		
		\item[\bf{(b)}] if $1-\beta<\alpha<\frac{1}{2}$,
		$\lim_{N\to\infty}U^{(N)}_q(out)=1$ and $\lim_{N\to\infty}U^{(N)}_q(in)=0$.
		\item[\bf{(c)}] if $\alpha=1-\beta< \frac{1}{2}$, $\lim_{N\to\infty}U^{(N)}_q(out)=\varepsilon_2(\alpha,\beta)\in(0,1)$ and $\lim_{N\to\infty}U^{(N)}_q(in)=0$.
		\item[\bf{(d)}] if $\alpha<\min\{\frac{1}{2},1-\beta\}$, $\lim_{N\to\infty}U^{(N)}_q(\star)=0, \star\in\{in,out\}.$
		\item[\bf{(e)}] if $\alpha=\frac{1}{2}<1-\beta$, $\lim_{N\to\infty}U^{(N)}_q(\star)= \varepsilon_1(\alpha,\beta)\in(0,1)$ , $\star\in\{in,out\}$.
		\item[\bf{(f)}] if $\alpha>\frac{1}{2}$, $\lim_{N\to\infty}U^{(N)}_q(\star)=1, \star\in\{in,out\}.$
	\end{itemize}
\end{theorem}

\begin{remark}{\bf (Anticommunities for negative $\beta$)}
	The above theorem is stated for arbitrary $\alpha\in\R$ and $\beta>0$. We notice that while for $\beta=0$ we are back to the complete graph with constant weight 1, for $\beta<0$, it would be more appropriate to speak about ``anticommunities'' rather than communities. In fact in this case, at every step, the RW prefers to change community rather than staying in its original one. Thus, it is somewhat artificial to see what the loop-erased partitioning  captures. This is the reason why the plot in~\cref{fig:phdiag} is restricted to $\beta\geq 0$. However, the theorem still remains valid for negative $\beta$ and, not surprisingly, the difference between the \emph{in} and \emph{out} potentials turns out to be zero.
\end{remark}

The next statement collects some simple consequences, deduced from these two-point LEP-potential, on the macroscopic structure of $\Pi_q$. We recall that $|\Pi_q|$ stands for the number of blocks in the random partition $\Pi_q$.

\begin{corollary}\label{macro}{\bf (Macroscopic emergent structure)}
	Under the assumption of~\cref{phasetrans}, the following scenarios hold true. 
	If $\beta>0$, there exists $c>0$ depending only on $\alpha$ and $\beta$ s.t.
	$$\P\left(|\Pi_q|=cN^{\alpha\wedge1}(1\pm o(1) ) \right)=1-o(1).$$
	Moreover:
	\begin{itemize}
		\item[\bf{(a)}] if $1-\beta<\alpha=\frac{1}{2}$ then $\whp$ there are two blocks of linear size s.t. each block has a fraction $(1-o(1))$ of vertices from the same community.
		\item[\bf{(b)}] if $1-\beta<\alpha<\frac{1}{2}$ then $\whp$ there are two blocks of size $N(1-o(1))$ s.t. each block has a fraction $(1-o(1))$ of vertices from the same community.
		\item[\bf{(c)}] if $\alpha=1-\beta< \frac{1}{2}$ then $\whp$ there is at least a block of linear size.
		\item[\bf{(d)}] if $\alpha<\min\{\frac{1}{2},1-\beta\}$ then $\whp$ there is one block of size $2N(1-o(1))$.
		\item[\bf{(e)}] if $\alpha=\frac{1}{2}<1-\beta$ then $\whp$ there is at least a block of linear size.
		\item[\bf{(f)}] if $\alpha>\frac{1}{2}$ then $\whp$ blocks of linear size do not exist.	
	\end{itemize}
\end{corollary}

\cref{phasetrans} says that the LEP--potential contains sufficient information to detect the underlying communities in a parametric region where the ratio of the {\em out} and {\em in} weights is bigger than $\sqrt{N}$. This suggests that estimating the probabilities in~\cref{FIP} could be a valuable method to design a community detection algorithm for well-separated regions. Nonetheless, there can be other  observables associated to $\Pi_q$ which perform better, meaning e.g. that they can be used for detection beyond regions ({\bf a})--({\bf c}) in~\cref{fig:phdiag}. 
However, it is not the scope of this paper to explore the practical applications and implications of this loop-erased partitioning in the context of community detection. For this reason we will omit complexity and other algorithmic considerations. As already mentioned, our main goal is rather to start understanding analytically the measure $\mu_q$ and its emergent structure. 

Our last result,~\cref{Rwdetection}, is the analogous of~\cref{phasetrans} for the RW-potential in~\cref{RWIP} and shows that this other potential gives essentially no insight on the emergent partition and very little can be detected from it. 
To state the result, we first give in the next lemma a characterization of the RW-potential which reveals that in reality this other 2-body interaction is determined only by the RW flow in the graph  rather than the LEP--measure.

\begin{lemma}\label{RWexpress}{\bf (RW--potential independent of LEP structure)}
For any arbitrary graph $G$ on $N$ vertices, the pairwise correlation function in~\cref{RWIP} admits the following representation:
$$ \overline{U}_q(x,y)=N^2 \left[K_q(x,x)K_q(y,y)-K_q(x,y)K_q(y,x)\right],$$
where  
$$K_q(x,y):= q(q-\mathcal{L})^{-1}(x,y)= \mathbb{P}_x(X(\tau_q)=y) $$
is, up to the factor $q$, the Green's kernel of the RW $X$ stopped at an independent exponentially distributed time $\tau_q$, with rate $q$.
\end{lemma}

We can now state the detectability captured by this RW--potential in the mean-field-community model. As for the LEP-potential we adapt the notation $\overline{U}_q( in/out )$ to distinguish between pairs within the same community or not.

\begin{proposition}\label{Rwdetection}{\bf (Detectability via RW--potential)}
Consider the two--community--graph $\mathcal{K}_{2N}(w_1,w_2)$ with  $w_1=1$, $w_2=N^{-\beta}$ and  $q=\Theta(N^{\alpha})$.
Then, if $\alpha\le0$ and $\beta>1-\alpha$
 $$ \overline{U}_{q}(\star)\sim \begin{cases}
4q^2+8 q&\text{ if } \star= in ,\\
4q^2+8 q+4 & \text{ if } \star= out.
\end{cases}$$
On the other hand:  


$$\overline{U}_{q}(in)\sim\overline{U}_{q}(out)\sim\begin{cases}
4q(q+1)&\text{ if } \alpha\le0 \text{ and } \beta<1-\alpha, \\
N^{\max\{2,2\alpha\}} & \text{ if } \alpha>0.
\end{cases}$$

\end{proposition}

As anticipated, this last statement shows that this RW-potential is less informative than the LEP one.
In particular, the detectable parametric region is narrower and corresponds to the triangle for $\alpha\leq 0$ in the detectable region depicted in~\cref{fig:phdiag}. 

\section{Proofs of~\cref{proporso}: homogeneous complete graph}\label{proofcomplete}

\subsection*{Proof of~\cref{orsoformula}}
For convenience, we consider a discretization of the continuous time Markov process with generator
\begin{equation}\label{def:lap}
\cL=\cA-\cD,\quad\text{ with }\quad \cA=w(\mathbf{1}\mathbf{1}'-I) \quad
 \text{ and }\quad\text{ with }  \cD=(n-1)wI.
\end{equation}
Set $L=\frac{1}{Nw}\cL$, so that $L=I-\frac{1}{N}\mathbf{1}\mathbf{1}'$
and the associated transition matrix is given by
\begin{equation}
P=I-L=\frac{1}{N}\mathbf{1}\mathbf{1}'
\end{equation}
If we consider the killing as an absorbing state within the state space of the Markov chain extended from $V$ to $V\bigcup\{\Delta\}$, $\Delta$ denoting this absorbing state, we get the adjacency matrix
\begin{equation}
\widehat \cA=\left(\begin{matrix}
w\mathbf{1}\mathbf{1}'&q\mathbf{1}\\
\mathbf{0}'&0
\end{matrix}\right),
\end{equation}
and generator
\begin{equation}
\widehat \cL= \widehat \cA-\widehat \cD,\qquad \widehat \cD=\left(\begin{matrix}
[(N-1)w+q]I&\mathbf{0}\\
\mathbf{0}'&0
\end{matrix}\right).
\end{equation}
We can then normalize it by setting
\begin{equation}
\widehat L=\frac{1}{Nw+q}\widehat \cL=\left(\begin{matrix}
\frac{w}{Nw+q}\mathbf{1}\mathbf{1}'-I&\frac{q}{Nw+q}\mathbf{1}\\
\mathbf{0}'&0
\end{matrix}\right)
\end{equation}
and get a discrete RW with transition matrix given by
\begin{equation}\label{discrete}
\widehat P= I-\widehat L=\left(\begin{matrix}
\frac{w}{Nw+q}\mathbf{1}\mathbf{1}'&\frac{q}{Nw+q}\mathbf{1}\\
\mathbf{0}'&1
\end{matrix}\right)=\left(\begin{matrix}
(1-p)\frac{1}{N}\mathbf{1}\mathbf{1}'&p\mathbf{1}\\
\mathbf{0}'&1
\end{matrix}\right),
\end{equation}
where 
\begin{equation}\label{p}
r:=\frac{q}{Nw + q}.
\end{equation}
It should be clear that a sample of a LE-path starting at a given vertex can be obtained as the output of the following procedure:
\begin{itemize}
	\item With probability $r$ the discrete process reaches the absorbing state. In particular we set $T_q$ for a geometric random variable of parameter $q/(Nw+q)$. 
	\item With probability $1-r$ the LERW moves accordingly to the law $P(v,\cdot)$ where $v$ is the last reached node.
	\item We call $H_n$ the vertices covered by the LE-path up to time $n$. 
Then, if at time $n+1$ the transition $X_n\to X_{n+1}$ takes place and the vertex $X_{n+1}\not\in H_n$, then $ H_{n+1}=H_n\cup\{X_{n+1} \}$. Conditioning on $|H_n|$, the latter event occurs with probability $\frac{N- H_n}{N}$.
Conversely, if $X_{n+1}\in H_n$, then we remove from $H_n$ all the vertices that has been visited by the LERW since its last visit to $X_{n+1}$. As consequence the quantity $|H|$ reduces. One can then compute that the reductions occur with law
	\begin{equation}
	\P\left( |H_{n+1}|=h\:|\:|H_n|\ge h, T_q>n+1 \right)=\frac{1}{N}.
	\end{equation}
\end{itemize}
It would be easier to look at the quantity $|H_{n}|$ by using the following metaphor. We interpret  $|H_{n}|$ as the height from which a bear fall down while moving on a stair of height $n$. In particular, we will assume that
\begin{itemize}
	\item The bear starts with probability 1 from the first stair.
	\item At each time the bear select a step of the stair uniformly at random, including also the step he currently stands on.
	\item If the choice made by the bear is a lower step (or the current one), he moves to that step.
	\item If he chooses an upper step, then he walks in the upper direction by a single step.
	\item Before doing each step, there is a probability $r$ as in~\cref{p} that the bear ``falls down''.
\end{itemize}

Let us next fix $q=0$, that is, $r=0$, so that we can study the bear's dynamic independently of his falling. By setting $Z(n)$ for the position of the bear at time $n\in \N$, we get
\begin{align}
\P(Z(0)=\cdot)=&\left(1,0,0,0,\dots,0\right)\\
\P(Z(1)=\cdot)=&\left(\frac{1}{N},1-\frac{1}{N},0,0,\dots,0\right)\\
\P(Z(2)=\cdot)=&\left(\frac{1}{N},\left(1-\frac{1}{N}\right)\frac{2}{N},\left(1-\frac{1}{N}\right)\left(1-\frac{2}{N}\right),0,\dots,0\right)\\
\P(Z(3)=\cdot)=&\left(\frac{1}{N},\left(1-\frac{1}{N}\right)\frac{2}{N},\left(1-\frac{1}{N}\right)\left(1-\frac{2}{N}\right)\frac{3}{N},\left(1-\frac{1}{N}\right)\left(1-\frac{2}{N}\right)\left(1-\frac{3}{N}\right),\dots,0\right)\\
\P(Z(n)=\cdot)=&\begin{cases}
\left(1-\frac{1}{N} \right)\left(1-\frac{2}{N}\right)\cdots \left(1-\frac{h-1}{N} \right)\frac{h}{N}&\text{ if }n\ge h\\
\left(1-\frac{1}{N} \right)\left(1-\frac{2}{N}\right)\cdots \left(1-\frac{h-1}{N} \right)&\text{ if }n=h-1\\
0&\text{ if }n<h -1.
\end{cases}
\end{align}
The latter implies that at time $n=h$ we reached the ergodic measure over the first $h$ steps of the stair, while at time $n=N$ the probability measure is exactly the ergodic one.

It is interesting to notice that an easier expression can be written for the cumulative distribution of the variable $Z(n)$, i.e.
\begin{equation}
\P\left\{Z(n)\ge h\right\}=\begin{cases}
\left(1-\frac{1}{N} \right)\left(1-\frac{2}{N}\right)\cdots \left(1-\frac{n-1}{N} \right)&\text{ if }n\ge h-1\\
0&\text{ if }n<h -1\\
\end{cases}
\end{equation}
Next, calling $T^-$ the time immediately before the bear falls, we get
\begin{align}
\nonumber\P\left\{Z(T^-)\ge\zeta \right\}=&\P\left\{T^-<h-1 \right\}\P\left\{Z(T^-)\ge h| T^-<n-1 \right\}+\P\left\{T^-\ge h-1 \right\}\P\left\{Z(T^-)\ge h| T^-\ge n-1 \right\}\\
=&0+(1-r)^{h-1} \left(1-\frac{1}{N} \right)\left(1-\frac{2}{N}\right)\cdots \left(1-\frac{h-1}{N}\right)
\end{align} 
which gives us the distribution of the last step of the bear before his failing. Recall that this is equivalent to the length of the  original LERW starting on $x\in \cK_{N}$, when the walk is stopped at an exponential time of rate $q$. Hence, we are now left to compute the probability that another walker, starting on $y\not= x$, is killed before it hits the previously sampled LERW.

Thanks to the bear metaphor, for the size of the LE-trajectory we get:
\begin{equation}
\P^{LE_q}_x(|\Gamma|\geq h)=(1-r)^{h-1}\prod_{i=1}^{h-1}\left(1-\frac{i}{N}\right)
\end{equation}
and by explicit computation, setting $T_\Gamma$ for the first hitting time of the LE-path $\Gamma$,
\begin{align*}
U^{(N)}_q(x,y)=&\sum_{h\geq 1}^{N-1}\P^{LE_q}_x(|\Gamma|= h)\P_y(T_q<T_\Gamma | |\Gamma|=h)\\
=&\sum_{h=1}^{N-1}\P^{LE_q}_x(|\Gamma|= h)[\P_y(T_q<T_\Gamma| |\Gamma|=h, y\in \Gamma)\P(y\in \Gamma ||\Gamma|=h)\\
&+\P_y(T_q<T_\Gamma||\Gamma|=h, y\notin \Gamma)\P(y\notin \Gamma| |\Gamma|=h)]\\
=&\sum_{h=1}^{N-1}\P^{LE_q}_x(|\Gamma|= h)\left(\frac{q}{q+hw}\right)\frac{N-h}{N-1}\\
=&\sum_{h=1}^{N-1}\P^{LE_q}_x(|\Gamma|\geq h)\left(\frac{q}{q+hw}\right)\frac{N-h}{N-1}-\sum_{h=1}^{N-1}\P_x(|\Gamma|\geq h+1)\left(\frac{q}{q+hw}\right)\frac{N-h}{N-1}\\
=&\sum_{h=1}^{N-1}\left[\left(\frac{Nw}{q+Nw}\right)^{h-1}\prod_{i=1}^{h-1}\left(1-\frac{i}{N}\right)\right]\left(\frac{q}{q+hw}\right)\frac{N-h}{N-1}+\\
&-\sum_{h=1}^{N-1}\left[\left(\frac{Nw}{q+Nw}\right)^{h}\prod_{i=1}^{h}\left(1-\frac{i}{N}\right)\right]\left(\frac{q}{q+hw}\right)\frac{N-h}{N-1}\\
=&\sum_{h=1}^{N-1}\frac{q}{q+Nw}\frac{N-h}{N-1}\left(\frac{Nw}{q+Nw}\right)^{h-1}\prod_{i=1}^{h-1}\left(1-\frac{i}{N}\right)\left[1-\frac{Nw}{Nw+q}\left(\frac{N-h}{N}\right)\right]\\
=&\sum_{h=1}^{N-1}\frac{q}{q+hw}\frac{N-h}{N-1}\left(\frac{Nw}{q+Nw}\right)^{h-1}\prod_{i=1}^{h-1}\left(1-\frac{i}{N}\right)\left(\frac{q+hw}{q+Nw}\right)\\
=&\sum_{h=1}^{N-1}\frac{q}{q+Nw}\left(\frac{Nw}{q+Nw}\right)^{h-1}\frac{N-h}{N-1}\prod_{i=1}^{h-1}\left(1-\frac{i}{N}\right)\\
=&\sum_{h=1}^{N-1}\frac{q}{q+Nw}\left(\frac{Nw}{q+Nw}\right)^{h-1}\prod_{i=2}^{h}\left(1-\frac{i}{N}\right)\\
=&\sum_{k=0}^{N-2}\frac{q}{q+Nw}\left(\frac{Nw}{q+Nw}\right)^{k}\prod_{i=2}^{k+1}\left(1-\frac{i}{N}\right).
\end{align*}
\qed

\subsection*{Proof of~\cref{orsolimite}}
Let
\begin{equation}
\frac{\xi_q}{N}:=\frac{q}{Nw+q}
\end{equation}
and notice that if $q=x\sqrt{N}$, with $x,w=\Theta(1)$, then
\begin{equation}
q=\frac{Nw\xi_q}{N-\xi_q}\Longrightarrow q\sim w\xi_q.
\end{equation}
Call
\begin{equation}
f(k,N):=\prod_{i=2}^k\left(1-\frac{i}{N}\right),
\end{equation}
in order to rewrite
\begin{align}
\begin{split}
U^{(N)}_q=&\sum_{k=0}^{N-2}\left(\frac{\xi_q}{N}\right)\left(1-\frac{\xi_q}{N} \right)^{k}\prod_{i=2}^{k+1}\left(1-\frac{i}{N}\right)\\
=&\sum_{k=0}^{N-2}\left(\frac{\xi_q}{N}\right)\left(1-\frac{\xi_q}{N} \right)^{k}f(k+1,N)
\end{split}
\end{align}
and notice that the first term in the latter sum is the probability that the geometric random variable $T_q \overset{d}{\sim} Geom\left(\frac{\xi_q}{N}\right)$ assumes value $k$. Moreover it trivially holds that
\begin{equation}\label{fknmin1}
f(k+1,N)\le1,\:\:\forall k\in\N,\qquad f(k+1,N)=0,\:\:\forall k\ge N-1.
\end{equation}
Hence,
\begin{equation}\label{uqmeant}
U^{(N)}_q=\E[f(T_q+1,N)].
\end{equation}

\noindent Let us approximate $\ln f(k+1,N)$ at the first order as follows
\begin{align}\label{eolo}
\begin{split}
\ln f(k+1,N)=&\sum_{i=2}^{k+1}\ln\left(1-\frac{i}{N}\right)=-\sum_{i=2}^{k+1}\frac{i}{N}+O\left(\frac{i^2}{N^2}\right)\\
=&-\frac{1}{N}\frac{(k+1)(k+2)-2}{2}+kO\left(\frac{k^2}{N^2}\right)=-\frac{1}{N}\frac{k^2+3k}{2}+O\left(\frac{k^3}{N^2}\right)\\
=&-\frac{k^2}{2N}+O\left(\frac{k}{N}+\frac{k^3}{N^2}\right)=:-\frac{k^2}{2N}+c_N(k).
\end{split}
\end{align}

Next, set $Y\,\overset{d}{\sim}\, exp(x)$ and $Z\,\overset{d}{\sim} \,\mathcal N(0,1)$, notice that $\E[e^{\frac{Y^2}{2}}]=\sqrt{2\pi}xe^{\frac{x^2}{2}}\P(Z>x)$ and that 
\begin{equation}
\lim_{N\to\infty}|\E[e^{-\frac{T_q^2}{2N}}]-\E[e^{\frac{Y^2}{2}}]|=0,
\end{equation}
since $T_q/\sqrt{N}$ converges in distribution to $Y$ as $N$ diverges.
In view of the latter together with~\cref{uqmeant}, we can estimate

\begin{align*}
\left|U^{(N)}_q-\sqrt{2\pi}xe^{\frac{x^2}{2}}\P(Z>x)
\right|&\leq 
\left| \E[f(T_q+1,N)]-\E[e^{-\frac{T_q^2}{2N}}]\right| + o(1) 
\\
\le &
\left| \E[f(T_q+1,N)]-\sum_{k=0}^{\lfloor N^\delta\rfloor}\P(T_q=k)e^{-\frac{k^2}{2N}}e^{c_N(k)}\right|
\\
&+
\left|\sum_{k=0}^{\lfloor N^\delta\rfloor}\P(T_q=k)e^{-\frac{k^2}{2N}}e^{c_N(k)}-\E[e^{-\frac{T_q^2}{2N}}]\right| 
+ o(1) 
\\\le&
\sum_{k=\lfloor N^\delta\rfloor+1}^\infty\P(T_q=k) + 	
\left|\sum_{k=0}^{\lfloor N^\delta\rfloor}\P(T_q=k)e^{-\frac{k^2}{2N}}e^{c_N(k)}-\sum_{k=0}^{\lfloor N^\delta\rfloor}\P(T_q=k)e^{-\frac{k^2}{2N}}\right|
+ o(1)\\
=& o(1),
\end{align*}
where the last inequality holds true by choosing any $\delta\in\left(\frac{1}{2},\frac{2}{3}\right)$ which in particular guarantees that $c_N(k)=o(1)$.
\qed


\section{Proofs for mean-field-communities}\label{proof2com}
\subsection{Proof of~\cref{2par2comssintpot}}
We use here the same line of argument used in the proof of~\cref{proporso}.
We will consider the process having state space $V=V_1\sqcup V_2$, where
$$V_1=\left\{1,\dots, N_1 \right\},\qquad V_2=\left\{N_1+1,\dots,N_1+N_2 \right\},$$
and generator
\begin{equation}\label{lap2com}
\cL(x,y)=\begin{cases}
w_1&\text{if } x\not=y  \text{ and } x,y \text{ in the same community}\\
w_2&\text{if } x\not=y  \text{ and } x,y \text{ not in the same community}\\
-(N_1-1)w_1-N_2 w_2&\text{if } x=y \text{ and } x\in V_1\\
-(N_2-1)w_1-N_1 w_2&\text{if } x=y  \text{ and } x\in V_2.
\end{cases}
\end{equation}
We will specialize later on the case $N_1=N_2=N$.\\
We now consider a killed LERW $\Gamma$, and we denote by $\Gamma_i$ the set of points of the $i$-th community belonging to $\Gamma$, i.e.,
\begin{equation}
\Gamma_i=\Gamma\cap V_i,\qquad i=1,2.
\end{equation}
We can write
\begin{equation}
\P_x^{LE_q}(|\Gamma_1|=k_1, |\Gamma_2|=k_2)=\sum_{\gamma: |\gamma_1|=k_1,|\gamma_2|=k_2}	\P_x^{LE_q}(\gamma)\label{marshall},
\end{equation}
and we assume, without loss of generality, that $x\in V_1$; then, by conditioning, we get for $ y\neq x$ with $y\in V_j$, $j=1,2$
\begin{equation}
U_q^{(N)}(x,y)=\sum_{k_1=1}^{N_1-\ind_{j=1}}\sum_{k_2=0}^{N_2-\ind_{j=2}}\P^{LE_q}_x(|\Gamma_1|=k_1, |\Gamma_2|=k_2)\cdot \P_y\left(T_q<T_{\Gamma}\big|\Gamma\right),\label{ma}
\end{equation}
$T_{\Gamma}$ being the hitting time of $\Gamma$.

\subsection*{The LERW starting from $x$}
A result due to Marchal~\cite{M00} provides the following explicit expression for the probability of a loop erased trajectory:
\begin{equation}\label{LERWlaw}
\P_x^{LE_q}(\Gamma=\gamma)=\prod_{i=1}^{|\gamma|}w(x_{i-1},x_i)\frac{\det_{V\setminus\gamma}{(qI+\mathcal{L})}}{\det{(qI+\mathcal{L})}}.
\end{equation}
By looking closely at the latter formula we distinguish two parts: a product over the weights of the edges of the path and an algebraic part containing the ratio of two determinants which encodes the ``loop-erased'' feature of the process. In particular we notice that the former contains all the details about the trajectory, while the latter only depends on the number of points visited in each community. Let $j_1$ (respectively, $j_2$)  be the number of jumps from the first community to the second (from the second to the first, respectively) along the LE-path. We have
\begin{equation}\label{2comm}
\begin{split}
\P_x^{LE_q}(|\Gamma_1|=k_1, &|\Gamma_2|=k_2|x\in V_1,\:y\in V_2)=\\
=&\sum_{\gamma: |\gamma_1|=k_1,|\gamma_2|=k_2}	\P_x^{LE_q}(\Gamma=\gamma)\\
=&\binom{N_1-1}{k_1-1}\binom{N_2-1}{k_2}\cdot(k_1-1)!(k_2)!\cdot\sum_{j_{1}=0}^{\min\{k_1,k_2\}}\sum_{j_{2}=j_{1}-1}^{j_{1}}\binom{k_1-1}{j_1-\ind_{j_{1}\neq j_{2}}}\binom{k_2-1}{j_{2}-\ind_{j_{1}=j_{2}}}\cdot\\
&\cdot w_1^{k_1+k_2-(j_1+j_2)-1}w_2^{j_1+j_2}q\frac{\det_{V\setminus\{1,2,\dots,k_1,N_1+1,N_1+2,\dots,N_1+k_2\}}(qI+\mathcal{L})}{\det(qI+\mathcal{L})}
\end{split}
\end{equation}
where 
\begin{itemize}
	\item The first binomial coefficients stays for the $k_1-1$ possible choices for the points in $G_1$ (one of those must be $x$) over the possible $N_1-1$ points of the first community (except $x$). In the second community we can choose any $k_2$ vertices over the possible $N_2-1$ vertices of the second community (except $y$).
	\item The factorials stay for the possible ordering of the nodes covered in each community. Notice that the path on the first community must start by $x$.
	\item We sum over all the possible jumps from the first community to the second, $j_1$, and from the second to the first, $j_2$ (notice that if $j_2$ must be equal or one smaller than $j_1$).
	\item For any choice over the product of the previous three terms we have a path that has probability as given by the Marchal formula.
\end{itemize}
In the case in which we condition on having both $x$ and $y$ in the same (first, say) community we have
\begin{equation}\label{2comm2}
\begin{split}
\P_x^{LE_q}(|\Gamma_1|=k_1, &|\Gamma_2|=k_2|x\in V_1,\:y\in V_1)=\\
=&\sum_{\gamma: |\gamma_1|=k_1,|\gamma_2|=k_2}	\P^{LE_q}_x(\Gamma=\gamma)\\
=&\binom{N_1-2}{k_1-1}\binom{N_2}{k_2}\cdot(k_1-1)!(k_2)!\cdot\sum_{j_{1}=0}^{\min\{k_1,k_2\}}\sum_{j_{2}=j_{1}-1}^{j_{1}}\binom{k_1-1}{j_1-\ind_{j_{1}\neq j_{2}}}\binom{k_2-1}{j_{2}-\ind_{j_{1}=j_{2}}}\cdot\\
&\cdot w_1^{k_1+k_2-(j_1+j_2)-1}w_2^{j_1+j_2}q\frac{\det_{V\setminus\{1,2,\dots,k_1,N_1+1,N_1+2,\dots,N_1+k_2\}}(qI+\mathcal{L})}{\det(qI+\mathcal{L})}.
\end{split}
\end{equation}
Namely, only the first combinatorial term changes.

\subsection*{The ratio of determinants}
In our \emph{mean-field} setup, the terms in~\cref{2comm} and~\cref{2comm2} coming from ~\cref{LERWlaw} can be explicitly computed.
We consider here the two communities case, i.e. $V=V_1\sqcup V_2$, where the communities possibly have different sizes, $|V_1|=N_1$ and $|V_2|=N_2$. Now, consider the matrix obtained by erasing $k_1$ ($k_2$) rows and corresponding columns in the first community (the second one, respectively) in $-\cL$. We are left with a square matrix made of two square blocks on the diagonal of size $N_1-k_1=:K_1$ (respectively $N_2-k_2=:K_2$). We will denote this matrix by
\begin{equation}
-M=
\begin{pmatrix}
d_1         & \cdots           &w_1         & w_2      & \cdots      & w_2 \\
\vdots             & \ddots     & \vdots                 & \vdots      & \ddots      & \vdots \\
w_1   & \cdots & d_1       &  w_2                & \cdots                &w_2                      \\
w_2            & \cdots    &w_2        &d_2           &\cdots        &w_1                 \\
\vdots & \ddots     & \vdots        &\vdots            &\ddots       & \vdots               \\
w_2           &\cdots      & w_2        &w_1              &w_1                & d_2\\
\end{pmatrix}=
\begin{pmatrix}
A_1    &B       \\
B^{Tr}   &A_2
\end{pmatrix},
\end{equation}
where the elements on the diagonal are given by
\begin{equation}
d_1=-((N_1-1)w_1 + N_2w_2),\qquad d_2=-((N_2-1)w_1 +N_1w_2).
\end{equation}
We want to find $K_1+K_2$ solutions of the problem
\begin{equation}
-Mv=\lambda v \label{eigen1}
\end{equation}
First we consider eigenvectors of the form
$v=(x_1,x_1,...,x_1,x_2,...,x_2)^{Tr}$, where the upper component has length $K_1$ and the lower one has length $K_2$. If we write explicitly~\cref{eigen1} we get the following linear system:
\begin{equation}\label{smalllumppro}
-\begin{pmatrix}
d_1 +(K_1-1)w_1 & K_2w_2 \\
K_1w_2 & d_2 + (K_2-1)w_1
\end{pmatrix} \begin{pmatrix}x_1\\x_2\end{pmatrix}=\lambda\begin{pmatrix}x_1\\x_2\end{pmatrix},
\end{equation}
from which we get two eigenvalues, which we will refer to as $\lambda_1$ and $\lambda_2$.
\\
Then we consider $v=(x_1,x_2,..., x_{K_1},0,...,0)^{Tr}$; with this choice we are left with the system
\begin{equation}
-\begin{pmatrix}
d_1       &\cdots      &w_1      \\
\vdots  &\ddots        &\vdots \\
w_1            &\cdots     &d_1
\end{pmatrix}
\begin{pmatrix}
x_1 \\
\vdots \\
x_{K_1}
\end{pmatrix}=\lambda\begin{pmatrix}
x_1 \\
\vdots \\
x_{K_1}
\end{pmatrix},
\qquad w_2(x_1+\cdots+x_{K_1})=0\end{equation}
and we have to find $K_1-1$ eigenvalues that are associated with eigenvector orthogonal to constants.
By direct computation, $A_1$ has eigenvalue $\lambda_1':=(N_1w_1+N_2w_2)$ with multiplicity $K_1-1$. With the opposite choice, namely $v=(0,...,0, x_1,..., x_{K_2})^{Tr}$, we get 
\begin{equation}
-\begin{pmatrix}
d_2       &\cdots      &w_1       \\
\vdots  &\ddots        &\vdots \\
w_1            &\cdots     &d_2
\end{pmatrix}
\begin{pmatrix}
x_1 \\
\vdots \\
x_{K_2}
\end{pmatrix}=\lambda\begin{pmatrix}
x_1 \\
\vdots \\
x_{K_2}
\end{pmatrix},
\qquad\qquad w_2(x_1+\cdots+x_{K_2})=0.
\end{equation}
Namely, there is an eigenvalue $\lambda_2':=(N_2w_1+N_1w_2)$ with multiplicity $K_2-1$.
So the spectrum of $M$ is
\begin{equation}
\text{spec}(M)=(\lambda_1, \lambda_2, \lambda_1',\lambda_2' )
\end{equation}
with multiplicity denoted by $\mu_{M}(\cdot)$:
\begin{equation}
\mu_{M}(\lambda_1)=1,\quad \mu_{M}(\lambda_2)=1,\quad \mu_{M}(\lambda_1')=K_1-1,\quad \mu_{M}(\lambda_2')=K_2-1.
\end{equation}
Therefore, we can see that the ratio of determinants in~\cref{2comm} and~\cref{2comm2} can be written explicitly. Indeed, at the denominator we have
\begin{equation}
\det{(qI+\mathcal{L})}=q(q+Nw_2)(q+N_1w_1+N_2w_2)^{N_1-1}(q+N_2w_1+N_1w_2)^{N_2-1},
\end{equation}
while at the numerator we are left with
\begin{equation}
\det_{V\setminus\{1,2,\dots,k_1,N_1+1,N_1+2,\dots,N_1+k_2\}}(qI+\mathcal{L})=(q+\lambda_1)(q+\lambda_2)(q+\lambda_1')^{N_1-k_1-1}(q+\lambda_2')^{N_2-k_2-1}
\end{equation}
where
\begin{equation}
\lambda_1':=N_1w_1+N_2w_2, \qquad \lambda_2':=N_1w_2+N_2w_1,
\end{equation}
while $\lambda_{1}$ and $\lambda_{2}$ are the two solutions of the system in~\cref{smalllumppro}.
In particular, if we specialize in the case $N_1=N_2=N$ we can conclude that the ratio of determinants is given by
\begin{equation}\label{def:theta}
\theta(k_1,k_2):=\frac{(q-\lambda_1(k_1,k_2))(q-\lambda_2(k_1,k_2))}{q(q+2Nw_2)(q+a)^{k_1+k_2}}
\end{equation}
where we defined
\begin{equation}
a:=N(w_1+w_2),
\end{equation}
and
\begin{equation*}
\lambda_{i}(k_1,k_2):=-\frac{1}{2}\left[w_1(k_1+k_2)+2Nw_2+(-1)^i\sqrt{w_1^2(k_1-k_2)^2+4\left(N-k_1\right)\left(N-k_1\right)w_2^2}\right],\quad i=1,2.
\end{equation*}
\subsection*{The path starting from $y$}
Now we have to consider the second path starting from $y$ which decides the root at which $y$ will be connected in the forest generated by the algorithm.
The latter corresponds to the second factor in~\cref{ma}. Notice that it is sufficient to consider such path in the simpler fashion, i.e. without erasing the loops, since we are only concerned with the absorption of the walker: either in $\gamma$ or killed at rate $q$. Moreover, we can exploit again the symmetry of the model to reduce it to a Markov chain $\bar X$ with state space $\{\bar 1,\bar 2,\bar 3,\bar 4\}$ corresponding to the sets $\left\{V_1\setminus \gamma_1, V_2\setminus \gamma_2, \gamma_1\sqcup \gamma_2,\Delta \right\}$, where $\Delta$ is again the absorbing state, i.e., the ``state-independent'' exponential killing. We will assume that 
$$|\gamma_i|=k_i,\qquad |V_i|=N_i,\qquad i=1,2.$$
Hence, the transition matrix we are interested in is given by
\begin{equation}\label{smallprocmorte}
\bar P:=\left(\begin{matrix}
Q&R\\0&I
\end{matrix} \right),
\end{equation}
where
\begin{equation}
Q:=D^{-1}\left(\begin{matrix}
\left(N_1-k_1-1  \right)w_1&\left(N_2-k_2 -1 \right)w_2\\
\left(N_1-k_1  \right)w_2&\left(N_2-k_2  \right)w_1
\end{matrix}
\right),
\end{equation}
\begin{equation}
D^{-1}:=\left(\begin{matrix}(q+a_1-w_1)^{-1}&0\\0&(q+a_2-w_1)^{-1}\end{matrix}\right),\qquad R:=D^{-1}\left(\begin{matrix}
k_1w_1+k_2w_2&q\\
k_1w_2+k_2w_1&q
\end{matrix}
\right).
\end{equation}
with
\begin{equation}
a_1:=N_1w_1+N_2w_2,\qquad a_2:=N_1w_2+N_2w_1.
\end{equation}
The states represent:
\begin{itemize}
	\item[($\bar 1$)] nodes of the $1^{st}$ community that have \emph{not} been covered by the LE-path started at $x$.
	\item[($\bar 2$)]  nodes of the $2^{nd}$ community that have \emph{not} been covered by the LE-path started at $x$.
	\item[($\bar 3$)]  nodes of \emph{both} communities that have been covered by the LE-path started at $x$.
	\item[($\bar 4$)]  the absorbing state $\Delta$.
\end{itemize}

Called $T_{abs}$ the hitting time of the absorbing set $\left\{\bar 3 ,\bar 4\right\}$, we want to compute the probability that the process $\bar X$ is absorbed in the  state, $\bar 4$ and not in $\bar 3$. In terms of our original process, this means that the process is killed before the hitting of the LE-path starting at $x$. By direct computation
\begin{align}\label{pmorte}
\begin{split}
\P_{\bar 2}(\bar X(T_{abs})=\bar 4)=&\sum_{k=0}^\infty\bar P^k(\bar 2,\bar 1)\frac{q}{q+a_1-w_1}+\sum_{k=0}^\infty\bar P^k(\bar 2,\bar 2)\frac{q}{q+a_2-w_1}\\
=&\left(\sum_{k=0}^\infty Q^k\right)D^{-1}\binom{q}{q}(2)\\
=&(I-Q)^{-1}D^{-1}\binom{q}{q}(2)\\
=:&P^{\dagger}(2)
\end{split}
\end{align}
notice that the first component of the vector $P^\dagger\in\R^2$  corresponds to the \emph{intra-community} case $\left\{x,y\right\}\in V_i$ for some $i$, i.e., $U^{(N)}_q(in)$, while the second one to the \emph{inter-community} case, namely $U^{(N)}_q(out)$.
\newline\newline
If we now use the assumption that $N_1=N_2=N$, the steps above allow us to write the following formulas
\begin{align}\label{vinter2ss}
\begin{split}
U^{(N)}_q(out)=&\sum_{k_1=1}^{N}\sum_{k_2=0}^{N-1}\binom{N-1}{k_1-1}\binom{N-1}{k_2}(k_1-1)!(k_2)!\theta(k_1,k_2)P^\dagger(2)\cdot\\
&\cdot\sum_{j_1=0}^{min(k_1,k_2)}\sum_{j_2=j_1-1}^{j_1}\binom{k_1-1}{f_1(j_1,j_2)}\binom{k_2-1}{f_2(j_1,j_2)}w_1^{k_1+k_2-1-j_1-j_2}w_2^{j_1+j_2}q
\end{split}
\end{align}
\begin{align}\label{vintra2ss}
\begin{split}
U^{(N)}_q(in)=&\sum_{k_1=1}^{N-1}\sum_{k_2=0}^{N}\binom{N-2}{k_1-1}\binom{N}{k_2}(k_1-1)!(k_2)!\theta(k_1,k_2)P^\dagger(1)\cdot\\
&\cdot\sum_{j_1=0}^{min(k_1,k_2)}\sum_{j_2=j_1-1}^{j_1}\binom{k_1-1}{f_1(j_1,j_2)}\binom{k_2-1}{f_2(j_1,j_2)}w_1^{k_1+k_2-1-j_1-j_2}w_2^{j_1+j_2}q
\end{split}
\end{align}
where
\begin{equation}
f_1(j_1,j_2):=j_1-\ind_{\left\{j_1\not=j_2 \right\}},\:\:\:f_2(j_1,j_2):=j_2-\ind_{\left\{j_1=j_2 \right\}},
\end{equation}
$\theta(k_1,k_2)$ as in~\cref{def:theta} and
\begin{equation}
P^\dagger=\frac{1}{q+a-w_1}(I-Q)^{-1}\binom{q}{q}.
\end{equation}
By direct computation we see that
\begin{equation}
P^\dagger=\frac{q}{c}\binom{q+k_2(w_1-w_2)+2w_2N}{q+k_1(w_1-w_2)+2w_2N}.
\end{equation}
where
\begin{equation}
c:=(q+k_1w_1)(q+k_2w_1)+Nw_2(2q+(k_1+k_2)w_1)+w_2^2[N(k_1+k_2)-k_1k_2].
\end{equation}
\subsection*{Local time interpretation}
Now consider the part of the formula concerning the jumps among the two communities of the killed-LE-path starting at $x$, i.e.
\begin{equation}\label{431}
\sum_{j_1=0}^{min(k_1,k_2)}\sum_{j_2=j_1-1}^{j_1}\binom{k_1-1}{f_1(j_1,j_2)}\binom{k_2-1}{f_2(j_1,j_2)}w_1^{k_1+k_2-1-j_1-j_2}w_2^{j_1+j_2}.
\end{equation}
The latter can be thought of as a function of a Markov Chain $(\tilde X_n)_{n\in\N}$ on the state space $\left\{\underline{1},\underline{2}\right\}$, with transition matrix
\begin{equation}\label{smallproc}
\tilde P=\left(\begin{matrix}p&1-p\\1-p&p \end{matrix}\right),\qquad p=\frac{w_1}{w_1+w_2}
\end{equation}
where the $\underline{i}$-th state stays for the $i$-th community. Indeed,  we can rewrite~\cref{431} as
\begin{equation*}
(w_1+w_2)^{k_1+k_2-1}\sum_{j_1=0}^{min(k_1,k_2)}\sum_{j_2=j_1-1}^{j_1}\binom{k_1-1}{f_1(j_1,j_2)}\binom{k_2-1}{f_2(j_1,j_2)}\left(\frac{w_1}{w_1+w_2}\right)^{k_1+k_2-1-j_1-j_2}\left(\frac{w_2}{w_1+w_2}\right)^{j_1+j_2}=
\end{equation*}
\begin{equation}\label{local}
=(w_1+w_2)^{k_1+k_2-1}\tilde \P_1(\ell(k_1+k_2)=k_1)
\end{equation}
with $\ell$ being the local time as in the statement of~\cref{2par2comssintpot}.

\subsection*{Geometric smoothing}
From the previous steps we get the following expression
\begin{align}
\begin{split}
U^{(N)}_q(out)=&\sum_{k_1=1}^{N}\sum_{k_2=0}^{N-1}\left(N-1\right)_{k_1-1}\left(N-1\right)_{k_2}\frac{(q-\lambda_1(k_1,k_2))(q-\lambda_2(k_1,k_2))}{q(q+2Nw_2)(q+a)^{k_1+k_2}}\cdot\\
&\cdot q(w_1+w_2)^{k_1+k_2-1}\tilde \P_1(\ell(k_1+k_2)=k_1)P^\dagger(2).
\end{split}
\end{align}
Next, we would like to make appear a geometric term as in the complete and uniform case of~\cref{proporso}. Notice that multiplying and dividing by $N^{k_1+k_2-1}$ one obtains
\begin{align}
\begin{split}
U^{(N)}_q(out)=&\sum_{k_1=1}^{N}\sum_{k_2=0}^{N-1}N^{-(k_1+k_2-1)}\left(N-1\right)_{k_1-1}\left(N-1\right)_{k_2}\frac{(q-\lambda_1(k_1,k_2))(q-\lambda_2(k_1,k_2))}{q(q+2Nw_2)}\cdot\\
&\cdot \frac{q}{q+a}\left(\frac{a}{q+a}\right)^{k_1+k_2-1}\tilde \P_1(\ell(k_1+k_2)=k_1)P^\dagger(2)
\end{split}
\end{align}
we can then define
\begin{equation}\label{xi}
\xi_{q,N}:=\frac{q}{q+a}=\frac{q}{q+N(w_1+w_2)}
\end{equation}
in order to obtain
\begin{align}\label{438}
\begin{split}
U^{(N)}_q(out)=&\sum_{k_1=1}^{N}\sum_{k_2=0}^{N-1}N^{-(k_1+k_2-1)}\left(N-1\right)_{k_1-1}\left(N-1\right)_{k_2}\frac{(q-\lambda_1(k_1,k_2))(q-\lambda_2(k_1,k_2))}{q(q+2Nw_2)}\cdot\\
&\cdot \P(T_q=k_1+k_2)\tilde \P_1(\ell(k_1+k_2)=k_1)P^\dagger(2)
\end{split}
\end{align}
and
\begin{align}\label{439}
\begin{split}
U^{(N)}_q(in)=&\sum_{k_1=1}^{N-1}\sum_{k_2=0}^{N}N^{-(k_1+k_2-1)}\left(N-2\right)_{k_1-1}\left(N\right)_{k_2}\frac{(q-\lambda_1(k_1,k_2))(q-\lambda_2(k_1,k_2))}{q(q+2Nw_2)}\cdot\\
&\cdot \P(T_q=k_1+k_2)\tilde \P_{\underline{1}}(\ell(k_1+k_2)=k_1)P^\dagger(1)
\end{split}
\end{align}
where $T_q$ is an independent random variable with law $Geom\left( \xi_{q,N} \right)$.
\subsection*{Conclusions}
One can ideally divide the formulas in~\cref{438,439}  in five terms, namely
\begin{enumerate}
	\item The entropic term
	\begin{equation}
	N^{-(k_1+k_2-1)}\left(N-2\right)_{k_1-1}\left(N\right)_{k_2}\qquad\text{ or }\qquad N^{-(k_1+k_2-1)}\left(N-1\right)_{k_1-1}\left(N-1\right)_{k_2}
	\end{equation}
	was already present in the complete and uniform case~\cref{orsoformula}. Indeed 
	\begin{equation}
	\prod_{h=2}^k\left(1-\frac{h}{N}\right)=N^{-(k-1)}(N-2)_{k-2}.
	\end{equation}
	\item The term related to the spectrum of the size 2 matrix presented in~\cref{smalllumppro}, i.e.
	\begin{equation}
	\frac{(q-\lambda_1(k_1,k_2))(q-\lambda_2(k_1,k_2))}{q(q+2Nw_2)}
	\end{equation}
	which is the same in both \emph{in} e \emph{out} community cases. It can be rewritten as the ratio between two parabolas in $q$, i.e.,
	\begin{equation}
	\frac{q^2+[(k_1+k_2)w_1+2Nw_2]q+(w_1+w_2)[(k_1+k_2)Nw_2+k_1k_2(w_1-w_2)]}{q^2+2Nw_2q}
	\end{equation}
	\item The term related to the geometric random variable of parameter $\xi_{q,N}$, which was present also in the case of the uniform graph,~\cref{orsoformula}.
	\item The term related to the local times of the 2-states Markov chain $\tilde P$, in~\cref{smallproc}.
	\item The term related to the absorption probability, i.e., to the quantity $P^\dagger$, see~\cref{pmorte}, as a function of the process $\bar{P}$ presented in~\cref{smallprocmorte}.
\end{enumerate}
It is worth noticing that the $P^\dagger$ above is slightly different from the $P^\dagger_\star$ in the statement of~\cref{2par2comssintpot} which contains the extra factor $\eta_\star$. 
At this point by setting
\begin{equation}
g'_{out}(k_1,k_2):=N^{-(k_1+k_2-1)}\left(N-1 \right)_{k_1-1}\left(N-1 \right)_{k_2}\frac{(q-\lambda_1(k_1,k_2))(q-\lambda_2(k_1,k_2))}{q(q+2Nw_2)}P^{\dagger}(2),
\end{equation}
\begin{equation}
g'_{in}(k_1,k_2):=N^{-(k_1+k_2-1)}\left(N-2 \right)_{k_1-1}\left(N \right)_{k_2}\frac{(q-\lambda_1(k_1,k_2))(q-\lambda_2(k_1,k_2))}{q(q+2Nw_2)}P^{\dagger}(1),
\end{equation}
we can write
\begin{align}\label{446} 
\begin{split}
U^{(N)}_q(out)=&\sum_{k_1=1}^{N}\sum_{k_2=0}^{N-1}g'_{out}(k_1,k_2)\P(T_q=k_1+k_2)\tilde \P_{\underline{1}}(\ell(k_1+k_2)=k_1)\\
=&\sum_{n=1}^{2N}\sum_{k_1+k_2=n}g'_{out}(k_1,k_2)\P(T_q=n)\tilde \P_{\underline{1}}(\ell(n)=k_1),
\end{split}
\end{align}
and
\begin{align}\label{447}
\begin{split}
U^{(N)}_q(in)=&\sum_{k_1=1}^{N-1}\sum_{k_2=0}^{N}g'_{in}(k_1,k_2)\P(T_q=k_1+k_2)\tilde{\P}_{\underline{1}}(\ell(k_1+k_2)=k_1)\\
=&\sum_{n=1}^{2N}\sum_{k_1+k_2=n}g'_{in}(k_1,k_2)\P(T_q=n)\tilde \P_{\underline{1}}(\ell(n)=k_1),\\
\end{split}
\end{align}
which is equivalent to the statement in~\cref{2par2comssintpot}.
\qed


\subsection{Proof of ~\cref{phasetrans}}\label{detect}

\noindent{
	\bf{ Proofs of \bf{(a)} and \bf{(b)}:  $1-\beta<\alpha<(=)\frac{1}{2}$ (detectability)  } } 

As expressed in the following lemma in this regime the RW is confined to its starting community for the entire life-time.

\begin{lemma}[RW is confined to its community up to dying]\label{lemmageom}
	Let $1>\alpha>1-\beta$ and for $x\in [2N]$, consider the event 
	
	$$E_x:=\{T_q>T_x^{out} \}$$
	
	where $T_x^{out}$ is the first time in which the RW moves out of the community in which $x$ lies. 
	
	Then, as $N\to\infty$, 
	$$\P_x(E_x)=o(1).$$
\end{lemma}
\begin{proof}
	Let $Z$ be a r.v. that can assume values in the set $\{Out, In, \Delta\}$ with probabilities:
	
	$$\P(Z=Out)= \frac{N^{1-\beta}}{N^\alpha+N+ N^{1-\beta}}=:a_N,$$
	$$\P(Z=In)= \frac{N}{N^\alpha+N+ N^{1-\beta}}=:b_N \quad \text{ and } \quad \P(Z=\Delta)=1- (a_N+b_N).$$

	Let $(Z_n)_{n\in\N}$ be a sequence of i.i.d. r.v.s with the same law of $Z$ and notice that 
	$$ \P(T_q<T_x^{out})=\P\left(\min\{n\ge0\:|\:Z_n=\Delta \}
	<\min\{n\ge0\:|\:Z_n=Out \}\right).$$
	Therefore 
	\begin{align*}
	\P_x(E_x)=\P_x(T_q>T_x^{out} )=&\sum_{n=1}^\infty\P_x(T_x^{out}=n,T_q>n)\\
	=&\sum_{n=1}^\infty b_N^{n-1}a_N\\
	=&\frac{a_Nb_N}{1-b_N}\sim N^{1-\beta-\alpha},
	\end{align*}
	from which the claim.
\end{proof}

In view of the decomposition in~\cref{LEdec} and the above lemma, we can write for any $x\neq y$

\begin{align}\notag
U^{(N)}_q(x,y)=&\sum_{\gamma}\P^{LE}_x(\gamma)\[\P_y(T_\gamma>T_q|E_x^c)\P_y(E_x^c)+\P_y(T_\gamma>T_q|E_x)\P_y(E_x)\]\\ \notag
=&o(1)+(1-o(1))\sum_{\gamma}\P_x^{LE}(\gamma)\P_y(T_\gamma>T_q|E_x^c)\\
\label{LEdetect}\sim&\sum_{\gamma}\P_x^{LE}(\gamma)\P_y(T_\gamma>T_q|E_x^c).
\end{align}	

Let us first consider  $U^{(N)}_q(out)$. In this case, by ~\cref{lemmageom}, 
for any $\alpha\leq 1/2$ and uniformly in $\gamma$, we have that
\begin{align*}
\P_y(T_\gamma<T_q|E_x^c)\le&\P_y(T_y^{out}<T_q|E_x^c)\\
=&\P_y(E_y)\\
=&o(1).
\end{align*}
As a consequence $\P_y(T_\gamma>T_q|E_x^c)\geq 1-o(1)$, and
by plugging this estimate in~\cref{LEdetect}, we get $U^{(N)}_q(out)\to 1$.

Concerning $U^{(N)}_q(in)$, one has to notice that, for every LERW $\gamma$ starting from $x$ and ending at the absorbing state, we can consider the event
$$E_{\gamma,y}=\{T_y^{out}<\min(T_\gamma,T_q) \}.$$

Once more, uniformly in $\gamma$, we get by~\cref{lemmageom} that
\begin{align*}
\P_y(E_{\gamma,y})\leq \P_y(E_y)=o(1)
\end{align*}
Thus, for $x,y \in [N]$, by~\cref{LEdetect}, we can estimate 
\begin{align*}
U^{(N)}_q(x,y)=&o(1)+(1-o(1))\sum_{\gamma}\P_x^{LE}(\gamma|E_x^c)\P_y(T_\gamma>T_q|E_x^c,E_{\gamma,y}^c)
\end{align*}
Notice that, under such conditioning, the sum can be read as the probability that two vertices in a complete graph with $N$ vertices end up in two different trees. Therefore, this reduces to~\cref{orsolimite}, which in turns gives $U^{(N)}_q(in)\to 0$ for $\alpha<1/2$ and $U^{(N)}_q(in)\to \varepsilon_0(\alpha)$ else.
\qed

\noindent{\bf{ Proof of {\bf(f)} :  $\alpha>\frac{1}{2}$ (high killing region)}} 

%

We will only show that $U^{(N)}_q(in)\to 1$, this will suffice since e.g. by direct computation one can check that $U^{(N)}_q(in)\geq U^{(N)}_q(out)$. 

Observe first that being $\alpha>\frac{1}{2}$, the length of the Loop-Erased path $\Gamma$ must be ``small'' with high probability. In particular we can bound
\begin{align*}
\P^{LE_q}_x\(|\Gamma|>\sqrt{N} \)\le& \P(T_q>\sqrt{N})\\
=&\(1-\frac{N^\alpha}{N+N^{1-\beta}+N^\alpha} \)^{\sqrt{N}}\\
=&o(1),
\end{align*}
hence
\begin{align*}
U^{(N)}_q(in)=&o(1)+\sum_{\gamma:\:|\gamma|\le\sqrt{n}}\P_x^{LE_q}(\Gamma=\gamma)\P_y(T_\gamma>T_q)\\
\ge&\sum_{\gamma:\:|\gamma|\le\sqrt{N}}\P_x^{LE_q}(\Gamma=\gamma)\frac{N^\alpha}{\sqrt{N}+N^\alpha}\\
=&1-o(1).
\end{align*}
\qed

We next prove the remaining items in~\cref{phasetrans} for which we will implement a similar strategy which we start explaining.
In all remaining regimes we need to show that $U^{(N)}_q(\star)$, $\star\in\{in,out\}$ either vanishes or stays bounded away from zero. 
To this aim, we will use the representation in~\cref{g}.

Depending on the parameter regimes, we will split the sum over $t$ in different pieces to be treated according to the asymptotic behavior of the involved factors. 
To simplify the exposition we will restrict in what follows to the positive quadrant $\alpha,\beta>0$. We stress however that, as the reader can check, the following estimates hold true and actually converge faster even outside of the positive quadrant. 

Let us start with a few observations. 
We notice that $\hat{f}(n,k)\leq 1$ for every choice of $k,N,n$, moreover $\hat{f}(t,n)=0$ if $n\ge N$.
Furthermore, for each $N$ , 
\begin{equation}\label{sum1}\sum_{n=1}^{\infty} \P(T_q=n)\sum_{k=1}^{n} \tilde \P_{\underline{1}}(\ell(n)=k)=\sum_{n=1}^{\infty} \P(T_q=n)=1,\end{equation} 
and while estimating the involved factors it will be crucial the behavior of the product  $\left(\hat{f} \theta P_{\star}^{\dagger}\right)(n,k)$ for which we can in general observe the following facts.

\begin{enumerate}[(A)]
	\item\label{b} For any $\varepsilon >0$, if $n>N^{1/2+\varepsilon}$, then it follows from~\cref{eolo} that $N\mapsto \hat{f}_N$ decays to zero, uniformly in $k$,  faster than any polynomial as $N\to \infty$.
	For such $n$'s , since $N\mapsto \theta_N P^\dagger_{\star}$ is polynomially bounded (uniformly in $n,k$), the contribution in~\cref{g} of such terms can be neglected.
	\item\label{c} Whenever we consider $n$'s for which 
	$ \theta P^\dagger_{\star}=o(1)$, because of~\cref{sum1} and the uniform control on $\hat{f}$, the contribution of such terms in~\cref{g} can also be neglected.
	\item\label{d} For $n$'s for which neither~\cref{b} nor~\cref{c} hold,  we will estimate the asymptotics of such part of the sum by controlling the mass of the geometric time $T_q$ against $\theta P^\dagger_{\star}$, and in the most delicate cases (on the separation lines in~\cref{fig:phdiag}), taking into account the behavior of the local time too.  
\end{enumerate}

We are now ready to treat the remaining parameter regimes using such facts.

\noindent{
	\bf{ Proof of {\bf(d)}:  $\alpha<\min\{\frac{1}{2},1-\beta\}$ (changing-communities before dying) } }

In this regime, the overall picture resembles the phenomenology of the complete graph. In particular, the RW will manage to change community before being killed and up to the killing time scale, it will forget its starting community. Moreover, with high probability a single tree of size $2N(1-o(1))$ will be formed, so that, given any two points $x,y$, they will end up in the same tree with high probability independently on their communities.

To prove the claim notice that, uniformly in $n,k$, 
\begin{equation}\label{Pblue}
P^\dagger_{\star}(n,k)\sim \frac{N^{1-\beta+\alpha} + N^{\alpha} k_\star}
{2N^{1-\beta+\alpha}+nN^{1-\beta}+k(n-k)}= \frac{N^{1-\beta+\alpha} }
{2N^{1-\beta+\alpha}+nN^{1-\beta}+k(n-k)}+ O\left( \frac{1}{N^{1-\beta-\alpha} } \right).
\end{equation}
As a consequence the asymptotics of $U^{(N)}_q(\star)$ will be independent of $\star$. 
To show that such a limit is zero we argue as follows.
Within this parameter region:
\begin{equation}
\theta(n,k)\sim 1+
\frac{nN^{\alpha} + 2k(n-k)}
{2N^{1-\beta+\alpha}}, 
\end{equation}
which together with~\cref{Pblue} leads to
\begin{align}\label{TPblue}
\nonumber\theta P^\dagger_{\star}(n,k)=& 
\frac{N^{1-\beta+\alpha}}
{2N^{1-\beta+\alpha}+nN^{1-\beta}+k(n-k)}
+\frac{k(n-k)}
{2N^{1-\beta+\alpha}+nN^{1-\beta}+k(n-k)}
+ O\left( \frac{k(n-k)}{N^{2(1-\beta)} }\right)
+ O\left( \frac{nN^{\alpha}}{N^{2(1-\beta)}} \right)\\
=:&\theta P^\dagger_I(n,k) +\theta P^\dagger_{II}(n,k)+\theta P^\dagger_{III}(n,k)+\theta P^\dagger_{IV}(n,k),\end{align}
We can now plug in this asymptotic representation of $\theta P^\dagger_{\star}$ in~\cref{g}, and separately treat the four resulting terms.

For the first term, namely the sum in~\cref{g} with $\theta P^\dagger_I$ in place of 
$\theta P^\dagger_{\star}$, we split the sum in $n$ into two parts at $N^{\alpha+\varepsilon}$, for small $\varepsilon>0$, and show that they both goes to zero, by using~\cref{d} and~\cref{c}, respectively
In fact, with this ``cut'' we see that: 
\begin{align}\label{formulavai1}
(I):= &\sum_{n=1}^{\infty} \P(T_q=n)\sum_{k=1}^{n} \tilde \P_{\underline{1}}(\ell(n)=k) \hat{f}(n,k) \theta P^\dagger_I(n,k)\\ =&\sum_{n<N^{\alpha+\varepsilon}} \P(T_q=n)\sum_{k=1}^{n} \tilde \P_{\underline{1}}(\ell(n)=k)\cdot 1\cdot \Theta(1)+\sum_{n\ge N^{\alpha+\varepsilon}}\P(T_q=n)\sum_{k=0}^{n} \tilde \P_{\underline{1}}(\ell(n)=k)\cdot 1 \cdot o(1)\\
=&\Theta\(\sum_{n<N^{\alpha+\varepsilon}} \P(T_q=n)\)+o(1)=o(1).
\end{align}

Analogously, for the second term we split the sum over $n$ into two parts at $N^{1/2+\varepsilon}$, with small $\varepsilon>0$.  Using~\cref{d} for the first part and~\cref{b} for the second one, we see that 
\begin{align}\label{formulavai2}
(II):= &\sum_{n=1}^{\infty} \P(T_q=n)\sum_{k=1}^{n} \tilde \P_{\underline{1}}(\ell(n)=k) \hat{f}(n,k) \theta P^\dagger_{II}(n,k)\\
=&\sum_{n<N^{1/2+\varepsilon}} \P(T_q=n)\sum_{k=1}^{n} \tilde \P_{\underline{1}}(\ell(n)=k)\cdot 1\cdot O(1)+o(1)\\
=&O\(\sum_{n<N^{1/2+\varepsilon}} \P(T_q=n) \)+o(1)\\
=&o(1).
\end{align}

For the third term we need to split the corresponding sum into three parts at  $T_1:=N^{1-\beta-\varepsilon}$ and $T_2:=N^{1/2+\varepsilon}$, which will be controlled by~\cref{c},~\cref{d} and~\cref{b}, respectively. That is
\begin{align}\label{formulavai3}
(III):= &\sum_{n=1}^{\infty} \P(T_q=n)\sum_{k=1}^{n} \tilde \P_{\underline{1}}(\ell(n)=k) \hat{f}(n,k) \theta P^\dagger_{III}(n,k)\\
\le&\sum_{n<T_1} \P(T_q=n)\sum_{k=1}^{n} \tilde \P_{\underline{1}}(\ell(n)=k)\cdot 1 \cdot o(1)+\sum_{n= T_1}^{T_2}\P(T_q=n))\sum_{k=1}^{n} \tilde \P_{\underline{1}}(\ell(n)=k)\cdot 1 \cdot O(N^{-1+2\beta+2\varepsilon})+o(1)\nonumber\\
=&o(1)+O\(N^{\alpha-\beta-\varepsilon}\cdot 1\cdot 1\cdot N^{-1+2\beta+2\varepsilon}\)+o(1)\\
=&o(1).
\end{align}

Finally, for the last term, we split the sum at 
$N^{1/2+\varepsilon}$. 
Indeed we see that: 
on the one hand, for $n\le N^{1/2+\varepsilon}$, we can use~\cref{d} since 
$$\theta P^\dagger_{IV}(n,k)=O\(N^{\frac{1}{2}+\varepsilon+\alpha-2(1-\beta)} \)\qquad\text{ and }\qquad\P\(T_q\le N^{\frac{1}{2}+\varepsilon}\)=O\(N^{-\frac{1}{2}+\alpha+\varepsilon} \).$$
On the other hand, for  $n\geq N^{1/2+\varepsilon}$, we can argue as in~\cref{b}.
Hence,
\begin{align}\label{formulavai4}
(IV):= &\sum_{n=1}^{\infty} \P(T_q=n)\sum_{k=1}^{n} \tilde \P_{\underline{1}}(\ell(n)=k) \hat{f}(n,k) \theta P^\dagger_{IV}(n,k)\\
\le&\sum_{n=1}^{N^{1/2+\varepsilon}} \P(T_q=n)\sum_{k=1}^{n} \tilde \P_{\underline{1}}(\ell(n)=k)\cdot 1\cdot      O\(N^{\frac{1}{2}+\varepsilon+\alpha-2(1-\beta)} \)+o(1)\\
=&O\(N^{-\frac{1}{2}+\alpha+\varepsilon} \cdot 1\cdot 1\cdot      N^{\frac{1}{2}+\varepsilon+\alpha-2(1-\beta)} \)+o(1)=o(1)
\end{align}
\qed

\noindent{\bf{Proofs of {\bf(c)} and {\bf(e)}  (high-entropy separating lines)}}

We start by proving {\bf(e)}, i.e.
\begin{equation}
\text{if } \alpha=\frac{1}{2}<1-\beta\Longrightarrow \exists \varepsilon>0\text{ s.t. }\lim_{N\to\infty}U^{(N)}_q(in)=U_q(out)=\varepsilon.
\end{equation}
Start noting that under our assumptions on $\alpha$ and $\beta$ we have that
\begin{equation}\label{thetagiallo}
\theta(n,k)\sim\frac{n\sqrt{N}+2N^{\frac{3}{2}-\beta}+2k(n-k)}{2N^{\frac{3}{2}-\beta}},
\end{equation}
and
\begin{equation}\label{pmortegiallo}
P_{\star}^\dagger(n,k)\sim\frac{k_\star\sqrt{N}+N^{\frac{3}{2}-\beta}}{2N^{\frac{3}{2}-\beta}+nN^{1-\beta}+k(n-k)}.
\end{equation}
We are going to split the sum over $n$ in~\cref{g} in three parts:
\begin{itemize}
	\item $n\le N^{\frac{1}{2}-\varepsilon}$. For such $n$'s we have that the product $\theta P_\star^\dagger(n,k)$ is of order $1$. Hence we can neglect this part by using~\cref{d} together with the estimate $$\P(T_q\le N^{\frac{1}{2}-\varepsilon})=O\(N^{-\frac{1}{2}-\alpha-\varepsilon} \).$$
	\item $ n> N^{\frac{1}{2}+\varepsilon}$. Also this part can be neglected thanks to the argument of~\cref{b}.
	\item $N^{\frac{1}{2}-\varepsilon}<n\le N^{\frac{1}{2}+\varepsilon}$. This is the delicate non-vanishing part. We start by noticing that, due to~\cref{thetagiallo} and~\cref{pmortegiallo}, the leading term in $\theta P_\star^\dagger$ does not involve $k_\star$, so that ---at first order--- $U^{(N)}_q(in)$ must equal $U^{(N)}_q(out)$. In order to show that the latter two are asymptotically bounded away from zero, we fix $c\in(0,1)$ and consider
	\begin{align}
	U^{(N)}_q(\star)\ge&\sum_{n=c\sqrt{N}}^{\sqrt{N}/c}\P(T_q=n)\sum_{k=1}^{n}\tilde \P_{\underline{1}}(\ell(n)=k)\theta(n,k)P^\dagger_\star(n,k)\hat{f}(n,k)\\
	\hat f=\Theta(1)\Rightarrow=&\Omega\( \sum_{n=c\sqrt{N}}^{\sqrt{N}/c}\P(T_q=n)\sum_{k=1}^{n}\tilde \P_{\underline{1}}(\ell(n)=k)\theta(t,k) P^\dagger_\star(n,k)\)\\
	\theta P^\dagger_\star(n,k)\in\[\frac{1}{2+c^{-1}},\frac{1}{2+c}\]\Rightarrow=&\Omega\(\sum_{n=c\sqrt{N}}^{\sqrt{N}/c}\P(T_q=n) \)=\Omega(1).\label{lessthan1}
	\end{align}
	Moreover, thanks to~\cref{lessthan1} we can easily deduce that the limit is strictly smaller than $\frac{1}{2}$.
\end{itemize}

We next conclude by giving the proof of {\bf(e)}, i.e., we are going to show that
\begin{equation}
\text{if } \alpha=1-\beta<\frac{1}{2}\Longrightarrow \exists \varepsilon>0\:\text{ s.t. }\lim_{N\to\infty}U^{(N)}_q(in)=0\:\text{ while }\lim_{N\to\infty}U^{(N)}_q(out)=\varepsilon.
\end{equation}

Observe that, under our assumptions on $\alpha$ and $\beta$, we have that
\begin{equation}
\theta(n,k)\sim\frac{3N^{2\alpha}+nN^\alpha+2k(n-k)}{3N^{2\alpha}},
\end{equation}
and
\begin{equation}
P_{\star}^\dagger(n,k)\sim\frac{N^{2\alpha}+k_\star N^\alpha}{3N^{2\alpha}+2nN^{\alpha}+k(n-k)},
\end{equation}
hence, their product behaves asymptotically as
\begin{equation}\label{formulavai5}
\theta P_{\star}^\dagger(n,k)=\Theta\(1+\frac{k_\star}{N^\alpha}\).
\end{equation}
To evaluate the asymptotic behavior of $U^{(N)}_q(\star)$, we split the sum over $n$ in~\cref{g} in three pieces:
\begin{itemize}
	\item $n\le N^{\alpha+\varepsilon}$: where, thanks to~\cref{formulavai5}, we know that  $\theta P_{\star}^\dagger(n,k)=O(N^\varepsilon)$. We argue as in~\cref{d}, obtaining
	\begin{align}
	\sum_{n\le N^{\alpha+\varepsilon}}\P(T_q=n)\sum_{k=1}^{n}\tilde \P_{\underline{1}}(\ell(n)=k)\theta(n,k)P^\dagger_\star(n,k)\hat{f}(n,k)\le&O\(N^\varepsilon\sum_{n\le N^{\alpha+\varepsilon}}\P(T_q=n)\)\\
	=&O\(N^{-1+2\alpha} \)
	\end{align}
	\item $n> N^{\frac{1}{2}+\varepsilon}$: in this case we can argue as in~\cref{b}.
	\item $N^{\alpha+\varepsilon}<n\le N^{\frac{1}{2}+\varepsilon}$: in this case we have to distinguish between $U^{(N)}_q(in)$ and $U^{(N)}_q(out)$.
\end{itemize}
Consider first $U^{(N)}_q(in)$. We call $E_n$ the following event concerning the Markov chain $(\tilde X_n)_{n\in\N}$
\begin{equation}
E_n:=\left\{\text{At least one jump occurs before time $n$}\right\}.
\end{equation}
Notice that if $N^{\alpha+\varepsilon}<n\le N^{\frac{1}{2}+\varepsilon}$ then the event $E_n^c$ occurs with high probability.
Hence, for any choice of $n\in[1,N]$ and $k\in[1,n]$ we can write
\begin{align}
\tilde\P_{\underline{1}}\(\ell(n)=k \)=&\tilde\P_{\underline{1}}(\ell(n)=k|E_n^c)\tilde\P_{\underline{1}}(E_n^c)+\tilde\P_{\underline{1}}(\ell(n)=k|E_n)\tilde\P_{\underline{1}}(E_n)
=\delta_{k,n}+o(1),
\end{align}
$\delta_{k,n}$ being the Kronecker delta.
Hence
\begin{align}
\sum_{n= N^{\alpha+\varepsilon}}^{N^{1/2+\varepsilon}}\P(T_q=n)\sum_{k=1}^{n}\tilde \P_{\underline{1}}(\ell(n)=k)\theta P^\dagger_{in}(n,k)\hat{f}(n,k)=&\Theta\(
\sum_{n= N^{\alpha+\varepsilon}}^{N^{1/2+\varepsilon}}\P(T_q=n)\sum_{k=1}^{n}\delta_{k,n}\(\frac{n-k}{N^\alpha}+1\)\)\\
=&\Theta\(
\sum_{n= N^{\alpha+\varepsilon}}^{N^{1/2+\varepsilon}}\P(T_q=n)\)=o(1).
\end{align}
Concerning $U^{(N)}_q(out)$, it is easy to get a lower bound via a soft argument by considering the events
\begin{equation}
B_x=\left\{\text{The LERW starting at $x$ never changes community} \right\}
\end{equation}
\begin{equation}
B'_y=\left\{\text{The RW starting at $y$ does not change community before dying} \right\}.
\end{equation}
Indeed,
\begin{align*}
U^{(N)}_q(out)\ge&\P\(B_x\)\P\(B'_y\)=\(\frac{N^\alpha}{N^\alpha+N^{1-\beta}}\)^2=\frac{1}{4}.
\end{align*}
Finally, we are left to show that $U^{(N)}_q(out)$ is asymptotically bounded away from $1$. We consider the further split
\begin{align*}
U^{(N)}_q(out)\le&o(1)+\sum_{n=N^{\alpha+\varepsilon}}^{\sqrt{N}}\P(T_q=n)\sum_{k=1}^{n}\tilde{\P}_{\underline{1}}(\ell(n)=k)(\hat f\theta P^\dagger_{out})(n,k)+\sum_{n=\sqrt{N}}^{N^{\frac{1}{2}+\varepsilon}}\P(T_q=n)\sum_{k=1}^{n}\tilde{\P}_{\underline{1}}(\ell(n)=k)(\hat f\theta P^\dagger_{out})(n,k).
\end{align*}
Focusing on the first sum in the latter display, thanks to~\cref{formulavai5}, we have that
\begin{align*}
\sum_{n=N^{\alpha+\varepsilon}}^{\sqrt{N}}\P(T_q=n)\sum_{k=1}^{n}\tilde{\P}_{\underline{1}}(\ell(n)=k)(\hat f\theta P^\dagger_{out})(n,k)\le&	\sum_{n= N^{\alpha+\varepsilon}}^{N^{1/2}}\P(T_q=n)\frac{n}{N^\alpha}+\sum_{n= N^{\alpha+\varepsilon}}^{N^{1/2}}\P(T_q=n)\\
=&\frac{1}{N}\sum_{n= N^{\alpha+\varepsilon}}^{N^{1/2}}\(1-\frac{1}{N^{1-\alpha}} \)^n+o(1)\\
\le&\frac{1}{N}\(\frac{\sqrt{N}(\sqrt{N}+1)}{2} \)\sim\frac{1}{2}.
\end{align*}
Concerning the second sum, we have
\begin{align*}
\sum_{n=\sqrt{N}}^{N^{\frac{1}{2}+\varepsilon}}\P(T_q=n)\sum_{k=1}^{n}\tilde{\P}_{\underline{1}}(\ell(n)=k)(\hat f\theta P^\dagger_{out})(n,k)=&O\(	\sum_{n=\sqrt{N}}^{N^{\frac{1}{2}+\varepsilon}}\P(T_q=n)\hat f(n,n)\frac{n}{N^\alpha} \)\\
=&O\(\frac{1}{N}\sum_{n=\sqrt{N}}^{N^{\frac{1}{2}+\varepsilon}}ne^{-\frac{n^2}{2N}}\)\\
=&O\(\frac{1}{\sqrt{N}}\sum_{m=1}^{N^{\varepsilon}}me^{-\frac{m^2}{2}}\)\\
=&O\(\frac{N^\varepsilon}{\sqrt{N}}\sum_{m=1}^{\infty}e^{-\frac{m^2}{2}} \)=o(1).
\end{align*}
\qed\\

\subsection{Proof of~\cref{macro}}
Let $0=\lambda_0\le \lambda_1\le\dots\le\lambda_{2N-1}$  be the eigenvalues of $-\cL$. As shown in\cite[Prop. 2.1]{AG}, the number of blocks of the induced partition, $|\Pi_q| $, is distributed as the sum of $2N$ independent Bernoulli random variables with success probabilities $\frac{q}{q+\lambda_i}$. That is
$$|\Pi_q|\overset{d}{\sim} \sum_{i=0}^{2N-1}X_i^{(q)},\qquad\text{
	with}\qquad X_i^{(q)}\overset{d}{\sim} Ber\(\frac{q}{q+\lambda_i} \),\quad i\in\left\{0,\dots,2N-1 \right\}$$
In case of the two-communities model we have
$$\lambda_0=0,\qquad\lambda_1=2N^{1-\beta},\qquad\lambda_i=N(1+N^{-\beta}),\quad i\in\left\{2,\dots,2N-1 \right\}.$$
Therefore
$$|\Pi_q|\overset{d}{\sim} 1+X+\sum_{i=1}^{2(N-1)}Y_i$$
where
$$X\overset{d}{\sim} Ber\(\frac{N^\alpha}{2N^{1-\beta}+N^\alpha} \)\qquad\text{and}\qquad Y_i\overset{d}{\sim} Ber\(\frac{N^\alpha}{N(1+N^{-\beta})+N^\alpha} \),\quad i\in\{1,\dots,2(N-1)\}.$$
Hence
$$\E|\Pi_q|\sim 1+ \frac{N^\alpha}{N^{1-\beta}+N^\alpha}+\frac{2N^{\alpha+1}}{N^\alpha+N}=\Theta(N^{\alpha\wedge 1}).$$
Moreover, we can prove the concentration result claimed in the first part of the statement by using the multiplicative version of the Chernoff bound on the sum of $Y_i$'s. Indeed, denoting by
$$S:=\sum_{i=1}^{2(N-1)}Y_i$$
we have that
$$\P\(\left|S-\E S \right|\ge \varepsilon\E S \)\le 2\exp\(-\frac{\varepsilon^2\E S}{3} \),$$
and since 
$$\E S\sim\frac{2N^{\alpha+1}}{N^\alpha+N}=\omega(1)$$
we can deduce the concentration of $|\Pi_q|$.\\
Notice also that the second part of the statement is a trivial consequence of the detectability result of~\cref{phasetrans}.
\qed

\subsection{Proof of~\cref{RWexpress}}
In this proof we will consider the probability measure $\nu_q$ on the space of rooted spanning forests studied in~\cite{AG}, namely,
\begin{equation}
\nu_q(F)=\frac{q^{|\rho(F)|}w(F)}{Z(q)}, \quad F\in \cF,
\end{equation}
where we denoted by $\rho(F)$ the set of roots of $F\in\cF$. As mentioned in~\cref{Wilson}, we stress that the measure in~\cref{LEP} can be obtained by projecting this forest measure $\nu_q(\cdot)$ on the set of partitions.

Call $\mathcal{B}_q$ the $\sigma$-field generated by the block structure $\Pi_q$ of the random forest $F$. By~\cite[Proposition 6.4]{AG}, we have
\begin{align}
\P\(x,y\in \rho(F)\bigg\rvert \cB_q \)=\ind_{\{B_q(x)\neq B_q(y)\}}\frac{\mu(x)\mu(y)}{\mu(B_q(x))\mu(B_q(y))}.
\end{align}
Now we notice that by~\cref{RWIP} and the tower property,
\begin{align}
\overline{U}_q(x,y)=\E\[\E\[\frac{\ind_{\{B_q(x)\neq B_q(y)\}}}{\mu(B_q(x))\mu(B_q(y))} \bigg\rvert \cB_q \] \]=\frac{1}{\mu(x)\mu(y)}\P\(x,y\in\rho(F)\).
\end{align}
We can now invoke~\cite[Theorem 3.4]{AG}, stating that the set of roots is a determinantal process with kernel $K_q$. As a consequence we obtain that
\begin{equation}
\P\(x,y\in\rho(F) \)=K_q(x,x)K_q(y,y)-K_q(x,y)K_q(y,x),
\end{equation}
and the claim readily follows. \qed

\subsection{Proof of~\cref{Rwdetection}}
We consider here the discrete time version of the process $X$ as presented in~\cref{proporso}, see~\eqref{discrete}. As a warm-up, we start by computing the potential in the complete graph with unitary weights. In this case, 
\begin{equation}
K_q(x,y)=
\delta_{x,y}\P(T_q=1)+\sum_{t\ge 1}\P_x\(X_t=y\:|\:T_q=t+1 \)\P(T_q=t+1),
\end{equation}
where
\begin{equation}
r_q:=\frac{q}{N+q} \qquad\text{ and } \quad \P(T_q=t+1)=r_q(1-r_q)^{t},\qquad\forall t\in \N_0.
\end{equation}
Therefore,
\begin{align}
K_q(x,y)= r_q\delta_{x,y}+\frac{1}{N}\sum_{t\ge 1}r_q(1-r_q)^t
=r_q\delta_{x,y}+\frac{1}{N}\(1-r_q\)=\frac{q\delta_{x,y}+1}{q+N}.
\end{align}
From which:
\begin{equation}\label{completeRWIP}
\overline{U}_q(x,y)= \(\frac{N}{q+N}\)^2\(q^2+2q\).
\end{equation}
Thus, in order to have a non-degenerate potential on $\mathcal{K}_N$, we need to take $q=\Theta(1)$.\qed 

We next move to the mean-field-community model $\mathcal{K}_{2N}(w_1,w_2)$ with $w_1=1$, $w_2=N^{-\beta},\beta>0$ and arbitrary $q$. The corresponding discrete-time RW is killed at an independent geometric time $T_q\overset{d}{\sim} Geom(r_q)$ with
\begin{equation}\label{rate}
r_q:=\frac{q}{N+N^{1-\beta}+q}.
\end{equation}

Denoting by $J_t$ the random variable that counts the number of times, up to time $t$, in which this random walk changes community, we notice that:
\begin{equation}
\P(J_t=k\:|\:\tau=t+1)=\binom{t}{k}(1-c)^{t-k}c^k,\qquad\forall k\in[0,t],
\end{equation}
that is, conditioning on $T_q=t+1$, $J_t$ has binomial distribution $ Bin(t,c)$ with success parameter
\begin{equation}
c:=\frac{N^{1-\beta}}{N+N^{1-\beta}}.
\end{equation}

We are now in shape to compute the probability that $x$ is absorbed in some $y$. Without loss of generality we assume $x\in[N]$, so that $y\in[N]$ and $y\in[2N]\setminus[N]$ determines the $in-$ and $out-$potential, respectively.

Thus:
   
\begin{align}\notag
K_q(x,y)=&\delta_{x,y}\P\(T_q=1\)+\sum_{t\ge 1}\P(T_q=t+1)\sum_{k\ge 0}\P_x(X_t=y\:|\:J_t=k;\:T_q=t+1)\P(J_t=k\:|\:T_q=t+1)\\
\label{latter}=&\delta_{x,y}r_q+\frac{1}{N}\sum_{t\ge 1}r_q(1-r_q)^t
\big[\ind_{y\in[N]}\P\(\text{Bin}(t,c)\in 2\N_0\)+\ind_{y\in[2N]\setminus[N]}\P\(\text{Bin}(t,c)\in 2\N_0+1\)\big]\\
\label{latter2}=&\delta_{x,y}r_q+O\(N^{-1}\),
\end{align}
where the last identity is due to the fact that the sum in~\cref{latter} is a probability and hence bounded above by $1$.

\noindent{{\bf (high killing)}
When $q=N^\alpha$, with $\alpha>0$, $r_q=\omega\(N^{-1}\)$, thus the $O\(N^{-1}\)$ term in~\cref{latter2} is negligible, and $\overline{U}_q(in/out)\sim N^2 r_q^2$. In particular, the potential diverges as $N^2$ or $N^{2\alpha}$ depending on $\alpha\geq 1$ or $\alpha<1$, respectively.

\noindent{{\bf (order one killing)} 
In the regime $q=O(1)$, the $O(N^{-1})$ term in~\cref{latter2} is no longer negligible and needs to be analyzed further. Let us first consider the sub-regime $q=\Theta(1)$.


Notice that, when $t=\Theta(1/r_q)$, 
\begin{equation}\label{333}
\E(\text{Bin}(t,c))=\frac{c}{r_q}=\frac{N^{1-\beta}}{q}=\begin{cases}
o(1)&\text{if }\beta>1\\
\omega(1)&\text{if }\beta<1.
\end{cases}
\end{equation}
Clearly, $\E(\text{Bin}(t,c))=o(1)$ implies that $\P\(\text{Bin}(t,c)\in 2\N_0 \)=1+o(1)$, while if $\E(\text{Bin}(t,c))=\omega(1)$ then $\P\(\text{Bin}(t,c)\in 2\N_0 \)=\frac{1}{2}+o(1).$
From which, if $\beta>1$, then
\begin{align}\label{111}
\sum_{t\ge 1}r_q(1-r_q)^t\P\(\text{Bin}(t,c)\in 2\N_0\)\sim&1,
\end{align}
while,  for $\beta<1$:
 \begin{align}\label{222}
 \sum_{t\ge 1}r_q(1-r_q)^t\P\(\text{Bin}(t,c)\in 2\N_0+1\)\sim \sum_{t\ge 1}r_q(1-r_q)^t\P\(\text{Bin}(t,c)\in 2\N_0\)\sim \frac{1}{2},
 \end{align}
 where in~\cref{111,222} we used the fact that, in order to compute the first order, it is sufficient to restrict the sum over $t$ to the values on the scale $\Theta(1/r_q)$.
By~\cref{latter} and the above estimates, we conclude that, for $\beta>1$: 
\begin{align}\label{res2}
K_q(x,y)\sim\begin{cases}
\frac{1}{N}&\text{if } y\in[N]\setminus \{x\}\\
\frac{t\cdot c}{N}=o(N^{-1})&\text{if } y\in[2N]\setminus[N],
\end{cases}
\end{align}
and $K_q(x,x)\sim\frac{q+1}{N}$, which together with~\cref{RWIP} lead to:

\begin{align}
&\beta>1\qquad\Longrightarrow\qquad \overline{U}_q(\star)\sim\begin{cases}
4q^2+8q&\text{if } \star= in\\
4q^2+8q+4&\text{if }\star=out
\end{cases}.
\end{align}

On the other hand, for $\beta<1$, the estimate in~\cref{222} shows that, regardless of the community of $y$, $K_q(x,y)\sim (\delta_{x,y}q+1/2)/N$. Thus the $in-$ and $out-$ potentials are asymptotically equivalent. In particular, $\overline{U}_q(in/out)\sim 4q^2+4q$.

\noindent{{\bf (vanishing killing)} 
It remains to analyze the case when $q=N^{\alpha}$ for some negative $\alpha<0$.
In this case, we have that\begin{equation}\label{444}
\E(\text{Bin}(t,c))=N^{1-\beta-\alpha}=\begin{cases}
o(1)&\text{if }1-\alpha<\beta\\
\omega(1)&\text{if }1-\alpha>\beta.
\end{cases}
\end{equation}
We can then argue as in the case $q=\Theta(1)$ but distinguishing between $\beta$ being bigger or smaller than $1-\alpha$. In particular, due to~\cref{444}, when $\beta<1-\alpha$ the resulting
$in-$ and $out-$ potentials are asymptotically equivalent and decay as $N^\alpha$.
On the other hand, for $\beta>1-\alpha >1$, $r_q\sim N^{\alpha-1}$, which together with~\cref{444} and~\cref{latter} lead to the estimates: $K_q(x,x)\sim r_q+N^{-1}\sim N^{-1}$, $K_q(x,y)\sim N^{-1}$ for $y\in[N]\setminus\{x\}$ and $K_q(x,y)=o(N^{-1})$ for pairs $(x,y)$ in different communities. By plugging these estimates in~\cref{RWexpress} the statement follows. \qed

\section*{{Acknowledgments}} 
{
L. Avena was supported by NWO Gravitation Grant 024.002.003-NETWORKS. M. Quattropani was partially supported by the INdAM-GNAMPA Project 2019 ``Markov chains and games on networks''.
Part of this work started during the preparation of the master thesis~\cite{Q16} and the authors are thankful to Diego Garlaschelli for acting as co-supervisor of this thesis project.
}

\end{document}